\documentclass[paper=a4,bibtotoc,parskip=half]{scrartcl}
\usepackage[left=3cm, right=3cm]{geometry}
\usepackage[english]{babel}
\usepackage{array}
\usepackage{caption}
\usepackage{amsmath}
\usepackage{amsfonts}
\usepackage{amsthm}
\usepackage{amssymb}
\usepackage{algorithm}
\usepackage{algorithmic}
\usepackage[font=small,labelfont=bf]{caption}

\usepackage{xspace}
\usepackage{euscript}
\usepackage{graphicx}
\usepackage{caption}
\usepackage{subcaption}
\usepackage{amscd}
\usepackage{tabularx}
\usepackage{lipsum}
\usepackage{array}
\usepackage{mwe}
\usepackage{graphbox}
\usepackage{calc}

\usepackage{enumerate}
\usepackage{epsfig} 
\usepackage{graphics} 
\theoremstyle{plain}
\newtheorem{theo}{Theorem}[section]
\newtheorem{lem}[theo]{Lemma}
\newtheorem{prop}[theo]{Proposition}
\theoremstyle{definition}
\newtheorem{definition}[theo]{Definition}
\theoremstyle{remark}
\newtheorem{rem}[theo]{Remark}
\numberwithin{equation}{section}

\usepackage{blindtext}
\title{{Reconstruction of a small acoustic inclusion \\via Time-dependent Polarization Tensors\thanks{\footnotesize This work was supported by the SNF grant 200021-172483.}}}
\author{{{Lorenzo Baldassari}} \and {{Andrea Scapin}}}
\date{}
\begin{document}
	\setcounter{section}{0}
	\setcounter{secnumdepth}{2}	
	\maketitle 
	\thispagestyle{empty}
	
\begin{abstract}
This paper aims at introducing the concept of time-dependent polarization tensors (TDPTs) for the wave equation associated to a diametrically small acoustic inclusion, with constitutive parameters different from those of the background and size smaller than the operating wavelength. Firstly, the solution to the Helmholtz equation is considered, and a rigorous systematic derivation  of a complete asymptotic expansion of the scattered field due to the presence of the inclusion is presented. Then, by applying the Fourier transform, the corresponding time-domain expansion is readily obtained after truncating the high frequencies. The new concept of TDPTs is shown to be promising for performing imaging. In particular, the optimization approach proposed in \cite{FDPTs} is extended to TDPTs. Numerical simulations are driven, showing that the TDPTs reconstructed from noisy measurements allow to image fine shape details of the inclusion.

	\medskip
	
	
	\medskip
	
	
\end{abstract}
	
\section{Introduction} 
%
%
Experimental data suggest that bats use temporal information for most, if not all, perceptual tasks. The bat perceives the phase of the sounds, which cover the 25- to 100-kilohertz frequency range, as these are represented in the auditory system after peripheral transformation. The acoustic image of a sonar target is apparently derived from \text{time-domain} or periodicity information processing by the nervous system \cite{Bat perception, Bat echolocation}. It appears natural then to extend the shape reconstruction and classification methods proposed in \cite{FDPTs, echolocation, target classification} to the time-domain.

Ammari et al. in \cite{FDPTs} introduced the concept of frequency-dependent polarization tensors (FDPTs) for a small inclusion. These tensors encode relevant information on the inclusion and appear naturally when we describe the perturbation of echoes emitted by animals such bats and dolphins \cite{Bat perception, Bat echolocation}. The extraction of the high-order FDPTs can be achieved from multi-static response (MSR) measurements \cite{FDPTs, echolocation}.

In this paper we deal with the problem of reconstructing a small acoustic inclusion by using the new concept of time-dependent polarization tensors (TDPTs) for the wave equation. The TDPTs can be interpreted as an extension of the concept of the FDPTs to the time-domain. 

Our aim in this paper is to model the problem of a static bat which is sending a wave and recording the scattering echoes due to the presence of an acoustic inclusion. This problem has to be evaluated in the frequency-domain first. Based on the layer potential techniques in \cite{layerpotentialtech}, we derive an asymptotic expansion for the scattered field in terms of the FDPTs. Such asymptotic expansion is based on careful and precise estimates of the remainders with respect to the frequency. In particular, in the two-dimensional case, thorough estimates are needed due to the logarithmic singular behavior of the Hankel function at the origin \cite{Abramowitz, Lebedev, Mat, Watson}. We require the inclusion to be small compared to the wavelength. In such a situation it is possible to expand the solution of the wave equation around the background solution \cite{boundary layer}. Recall that high frequencies correspond exactly to small wavelengths. The idea is then to truncate the high frequencies, as in \cite{Transient, Numerical Transient}. This corresponds to the case of a constant frequency (CF) bat, which cannot hear all the frequencies outside a certain range of finite values \cite{Bat perception}. By applying the truncated Fourier transform to the frequency-dependent asymptotic expansion, we switch to the time-domain. The TDPTs are defined as the building blocks of the corresponding time-dependent asymptotic expansion. Note that the leading-order term in this asymptotic formula has been derived by {Ammari et al.} \cite{Transient,Numerical Transient} for the three-dimensional case.


This paper is organized as follows. In Section 2, we recall the definitions of the boundary layer potentials in dimension $d = 2,3$, and state some basic results that are used throughout the paper. In Section 3, we describe the mathematical model concerning the Helmholtz equation, which we rewrite as a transmission problem, providing a representation formula for the solution. In Section 4, 5 and 6 we perform the derivation of asymptotic expansions in dimension $d = 2$. In particular, Section 4 is devoted to the proof of a stability estimate for the two-dimensional transmission problem, which is used in Section 5 to estimate the remainder of the expansion obtained in the frequency-domain, in terms of the operating frequency $\omega$. In Section 6, an expansion for the two-dimensional transient wave equation is presented. This asymptotic formula in time-domain is written in terms of the new concept of time-dependent polarization tensors (TDPTs). In Section 7, we show that high-order TDPTs allow to reconstruct both the volume and the material property of a small inclusion. It is worth mentioning that, in our framework, these information can be separated and retrieved without using a near-field expansion \cite{Transient, Numerical Transient}. Furthermore, we adapt a well-known procedure for reconstructing fine shape details of the inclusion by using the TDPTs. This algorithm consists of a recursive optimization of a functional based on its shape derivative \cite{GPT, GPT2, FDPTs}. 

In Section 8, we perform numerical experiments in the two-dimensional case to validate the usage of the TDPTs. The reconstruction procedures of Section 7 are tested for different inclusions and acquisition settings, and the results are reported. In particular, we observe that the optimization algorithm performs well in recovering the boundary of the inclusion even with moderate level of noise.


The results analogous to that of Sections 4, 5 and 6 are presented for the three-dimensional case in Appendix \ref{appendix1}.

\vspace{1mm}

\section{Preliminary results} 
Before introducing our problem we recall some basic facts about the boundary layer potentials that we use repeatedly in the sequel. 

Let us denote by $\Gamma_{\omega}$ the outgoing fundamental solution to the Helmoltz operator $\Delta +\omega^2$ in  $\mathbb{R}^d$, that is \cite{biomedicalimaging}
\begin{equation} \label{eq:fund_sol} \Gamma_{\omega}(x) := \begin{cases} \displaystyle \frac{i}{4} H_0^{(1)}(\omega|x|) ,  & d = 2, \\ \displaystyle \frac{e^{i\omega|x|}}{4\pi |x|} ,  & d = 3 . \end{cases} \end{equation}
Here, $H_0^{(1)}$ denotes the Hankel function of the first kind of order zero. We also consider $\Gamma_0$ defined by
\begin{equation*} \label{eq:fund_sol_laplace} \Gamma_{0}(x) := \begin{cases} \displaystyle   \frac{1}{2\pi} \log |x| ,  & d = 2 ,\\ \displaystyle \frac{1}{4 \pi | x |} , & d = 3. \end{cases} \end{equation*}
Note that $\Gamma_\omega$ solves (in the sense of distributions) the equation 
\[ (\Delta + \omega^2)\Gamma_{\omega} = - \delta_0 \quad \mbox{ in } \mathbb{R}^d , \quad d = 1, 2 .\]

Let $D$ be a bounded Lipschitz domain in $\mathbb{R}^d$. The single- and double-layer potentials on $D$, $\mathcal{S}_{D}^{\omega}$ and $\mathcal{D}_{D}^{\omega}$, are defined as follows:  $\phi \in L^2(\partial D)$,  
\begin{align*}
\mathcal{S}_D^{\omega}[\phi](x) &= \int_{\partial D} \Gamma_{\omega} (x-s) \phi(s) \; d\sigma(s),\\
\mathcal{D}_D^{\omega}[\phi](x) &= \int_{\partial D} \frac{\partial\Gamma_{\omega}}{\partial \nu_s} (x-s) \phi(s) \; d\sigma(s) .
\end{align*}
 The behavior of $\mathcal{S}_D^{\omega}[\phi]$ across the boundary $\partial D$ is described by the following well-known formulas \cite{math and stat method}.
\begin{lem}	For $\phi \in L^2(\partial D)$,
\begin{align*}
	\mathcal{S}^\omega_D[\phi]|_{+} (x) &=\mathcal{S}^\omega_D[\phi]|_{-} (x), \  \text{for a.e. } \ x \in \partial D, 
\\
	\left. \frac{\partial \mathcal{S}^\omega_D[\phi]}{\partial \nu} \right|_{\pm} (x)& = \left(\pm \frac{1}{2}\mathcal{I} +(\mathcal{K}^\omega_D)^* \right) [\phi] (x), \ \text{for a.e. } \ x \in \partial D, 
\end{align*}
	where $\mathcal{I}$ is the identity operator and $(\mathcal{K}^\omega_D)^*$ is defined by
	$$
	(\mathcal{K}^\omega_D)^*[\phi](x) = \int_{\partial D} \frac{\partial \Gamma_\omega (x-y)}{\partial \nu_x} \phi(y) \; d\sigma(y).
	$$  
	Note that $(\mathcal{K}^\omega_D)^*$ is the $L^2$-adjoint of $\mathcal{K}^\omega_D$, with
	$$
	\mathcal{K}^\omega_D[\phi](x) = \int_{\partial D} \frac{\partial \Gamma_\omega (x-y)}{\partial \nu_y} \phi(y) \; d\sigma(y).
	$$  
\end{lem}

We also recall the following lemma \cite{boundary layer, Low freq, Interface}.
\begin{lem}\label{t0}
	There exists $\varepsilon_0>0$ such that for $\varepsilon \omega<\varepsilon_0$,
	
	\begin{enumerate}\item[a)] if $d = 2$
	\begin{align*}
	\left\|\mathcal{S}_{B}^{\varepsilon \omega} [\phi] - \mathcal{S}^0_{B}[\phi] - \beta_{\varepsilon \omega} \int_{\partial B} \phi \right\|_{H^1(\partial B)} &\leq C (\varepsilon \omega)^2 \ln (\varepsilon \omega) \| \phi \|_{L^2(\partial B)},
	\\
	\left\|
	\left. \frac{\partial \mathcal{S}_B^{\varepsilon \omega}[{\phi}]}{\partial \nu}\right|_{\pm} - 	\left. \frac{\partial \mathcal{S}_B^{0}[{\phi}]}{\partial \nu}\right|_{\pm}\right\|_{L^2(\partial B)} &\leq C (\varepsilon \omega)^2 \ln (\varepsilon \omega) \| \phi \|_{L^2(\partial B)}, 
		\label{eq:estimates-SL-2D} 
	\end{align*}
	where $\beta_{\varepsilon \omega} = \frac{1}{2\pi} (\ln(\varepsilon \omega) - \ln 2 + \gamma -\frac{\pi}{2} i)$, and $\gamma$ is the Euler constant, 
	
\item[b)] if $d = 3$	\begin{align*}
	\|\mathcal{S}_B^{\varepsilon \omega} [{\phi}] - \mathcal{S}_B^{0}[{\phi}]\|_{H^1(\partial B)}& \leq C \varepsilon \omega \|{\phi}\|_{L^2(\partial B)}, \\
	\left\|
	\left. \frac{\partial \mathcal{S}_B^{\varepsilon \omega}[{\phi}]}{\partial \nu}\right|_{\pm} - 	\left. \frac{\partial \mathcal{S}_B^{0}[{\phi}]}{\partial \nu}\right|_{\pm}\right\|_{L^2(\partial B)} & \leq C \varepsilon \omega \|{\phi}\|_{L^2(\partial B)}.
	\end{align*}
\end{enumerate}
\end{lem}

Note that when $d = 2$ the single-layer potential is not, in general, invertible nor injective.


\section{Problem formulation}

\label{sec:problem-formulation}
Let $B$ be a bounded Lipschitz domain in $\mathbb{R}^d$ ($d = 2, 3$) such that $B$ contains the origin and $|B|=1$, and let $D = z+\varepsilon B$ be a small acoustic inclusion of contrast $k > 0$, $k \ne 1$, with $0<\varepsilon<1$. Let $\Gamma_{\omega}$ be as in \eqref{eq:fund_sol}. Denote by $V_y$ the field corresponding to a time-harmonic wave generated at $y \in \mathbb{R}^d \setminus \overline{D}$,
$$V_y(x,\omega) := \Gamma_{\omega} (x-y) ,$$
where $\omega > 0$ is the operating frequency, and $x \neq y$.  Moreover, we assume that $\operatorname{dist}(y,D) \gg 1$, that is, the inclusion $D$ is far from the source. 

Let $v_y$ be the field perturbed by the presence of the inclusion $D$, which is the solution to
\begin{equation} 
\nabla \cdot ( \chi(\mathbb{R}^d \setminus \overline{D} ) + k \chi(D)) \nabla v_y + \omega^2 v_y = - \delta_y,
\label{AE:49}
\end{equation}
with $\chi$ denoting the characteristic function, and $v_y-V_y$ satisfying the so-called Sommerfeld radiation condition, i.e., 
\begin{equation} 
\lim_{r \to +\infty} r^{(d-1)/2} \left ( \frac{\partial}{\partial r} (v_y-V_y) - i \omega (v_y-V_y) \right ) = 0, \quad |x| = r.
\label{eq:Sommerfeld}
\end{equation}

 Equation \eqref{AE:49} can be written equivalently as the following transmission problem
\begin{equation*}
\begin{cases}
\displaystyle \Delta v_y + \omega^2 v_y = - \delta_y, & \mathbb{R}^d \setminus \overline{D}, \\
\displaystyle \Delta v_y + \frac{\omega^2}{k} v_y = 0, & D, \\
v_y|_{+} = v_y|_{-}, & \partial D, \\
\displaystyle \frac{\partial v_y}{\partial \nu}\biggr |_{+} = k \frac{\partial v_y}{\partial \nu} \biggr |_{-}, & \partial D, \\
v_y-V_y \text{ satisfies condition } \eqref{eq:Sommerfeld} .\end{cases} 
\label{AE:50}
\end{equation*}
Notice that the solution $v_y$ can be represented as follows \cite{boundary layer}:
\begin{equation}
v_y (x,\omega)= 
\begin{cases}
V_y(x,\omega) + \mathcal{S}_D^{\omega} [\psi](x), & x\in \mathbb{R}^d \setminus \overline{D}, \\
\mathcal{S}_D^{\frac{\omega}{\sqrt{k}}}[\phi](x), & x\in D,
\end{cases} 
\label{AE:51} 
\end{equation}   
where the pair $(\phi, \psi) \in L^2(\partial D) \times L^2(\partial D)$ is the unique solution of the following system of boundary integral equations on $\partial D$:
\begin{equation}
\begin{cases}
\mathcal{S}_D^{\frac{\omega}{\sqrt{k}}} [\phi] - \mathcal{S}_D^{\omega}[\psi] = V_y,\\
\displaystyle k \frac{\partial \mathcal{S}_D^{\frac{\omega}{\sqrt{k}}} [\phi]}{\partial \nu} \biggr |_{-} - \frac{\partial \mathcal{S}_D^{\omega}[\psi]}{\partial \nu} \biggr |_{+} = \frac{\partial V_y}{\partial \nu},
\end{cases}
\text{ on } \partial D.
\label{AE:52}
\end{equation}

\begin{rem} The system \eqref{AE:52} has a unique solution provided that $\omega^2$ is not a Dirichlet eigenvalue for $-\Delta$ on $D$.	
	This is certainly true when $\lambda_1(D) \geq \left(1/|D|\right)^{2/d} C_d^{2/d} j_{d/2-1,1}$, where $j_{m,1}$ is the first positive zero of the Bessel function $J_m$, $C_d$ is the volume of the $d$-dimensional unit ball, and $\lambda_1(D)>0$ is the smallest eigenvalue of $-\Delta$ on $D$, see \cite{Iso, Krahn}.	
\end{rem}
Hereinafter, we limit our considerations to the two-dimensional problem only. For the three-dimensional case see Appendix \ref{appendix1}.

\vspace{1mm}

\section{Stability estimates for the transmission problem} 

\label{sec:stability-estimates-2d}

The technical estimates contained in this section will be used in the derivation of the asymptotic expansion to make the dependence of the reminder on the operating frequency explicit.

Let $D = \varepsilon B +z$, $|B|=1$ and $D \subset \mathbb{R}^2$. We suppose that $\omega \in (0, \varepsilon^{-\gamma})$, with $0<\gamma <1$. Note that there exists $\varepsilon_0 >0$ such that, for $\varepsilon$ sufficiently small, $\varepsilon \omega \leq \varepsilon_0<1$, i.e., $\varepsilon \omega$ can be made arbitrarily small. 
The main estimate is given by the following proposition.
\begin{prop}\label{t1}
	For each $(F,G)$ $\in H^1(\partial D) \times L^2(\partial D)$, let $(\phi, \psi) \in L^2(\partial D) \times L^2(\partial D)$ be the unique solution of the system of integral equations
	\begin{equation}
	\begin{cases}
	\mathcal{S}_D^{\frac{\omega}{\sqrt{k}}} [\phi] - \mathcal{S}_D^{\omega}[\psi] = F,\\
	\displaystyle k \left. \frac{\partial \mathcal{S}_D^{\frac{\omega}{\sqrt{k}}} [\phi]}{\partial \nu}\right|_{-} - \left. \frac{\partial \mathcal{S}_D^{\omega}[\psi]}{\partial \nu} \right|_{+} = G,
	\end{cases}
	\text{ on } \partial D.
	\label{1}
	\end{equation}
	Let $\varepsilon_0>0$ be such that $\varepsilon \omega \leq \varepsilon_0<1$. We have
	\begin{equation}
	\begin{aligned} 
	\|\phi \|_{L^2(\partial D)} + \|\psi \|_{L^2(\partial D)} \leq C (\varepsilon^{-1} \| F \|_{L^2(\partial D)} + \| \nabla F \|_{L^2(\partial D)} + \| G  \|_{L^2(\partial D)} ),
	\end{aligned} 
	\label{2}
	\end{equation} 
	where $C$ does not depend on $\varepsilon$ and $\omega$. 
\end{prop}
Proposition \ref{t1} states that the solution to system \eqref{1} depends continuously on the right-hand side of the system, i.e., $(F, G)$.

Since the two-dimensional fundamental solutions $\Gamma_{\varepsilon\omega}(x-y)$ and $\Gamma_{\varepsilon \omega /\sqrt{k}}(x-y)$ do not converge to $\Gamma_{0}(x-y) = \frac{1}{2\pi} \log |x-y|$ as $\varepsilon$ goes to zero, the proof of Proposition \ref{t1} is not immediate. The proof we present here relies on the following lemma \cite{boundary layer, Low freq, Interface}.

\begin{lem}\label{t2} 

\bigskip 
For each $(f,g)$ $\in H^1(\partial B) \times L^2(\partial B)$, let $(\phi, \psi) \in L^2(\partial B) \times L^2(\partial B)$ be the unique solution of the system
\begin{equation}
\begin{cases}
\displaystyle \mathcal{S}_B^0 [\phi] + \beta_{\varepsilon \omega / \sqrt{k}} \int_{\partial B} \phi - \mathcal{S}_B^0[\psi] - \beta_{\varepsilon \omega} \int_{\partial B} \psi = f,\\
\displaystyle k \left.\frac{\partial \mathcal{S}_B^{0} [\phi]}{\partial \nu}\right|_{-} - \left.\frac{\partial \mathcal{S}_B^{0}[\psi]}{\partial \nu}\right|_{+} = g,
\end{cases}
\label{sys0}
\text{ on } \partial B.
\end{equation}
Suppose there exists $\varepsilon_0>0$ such that $\varepsilon \omega \leq \varepsilon_0<1$. We have
\begin{equation}
\|\phi\|_{L^2(\partial B)} + \|\psi\|_{L^2(\partial B)} \leq C(
\|f\|_{H^1(\partial B)} + \|g\|_{L^2(\partial B)}), 
\label{est-lemma}
\end{equation}
where $C$ does not depend on $\varepsilon$ nor $\omega$.  
\end{lem} 

\begin{proof} 
	We define 
	$$
	\widehat{\mathcal{S}}^{\varepsilon \omega}_B[\phi] := \mathcal{S}^0_B[\phi] + \beta_{\varepsilon \omega} \int_{\partial B} \phi.
	$$ 
	Hariharan and MacCamy proved that $\widehat{\mathcal{S}}^{\varepsilon \omega}_B$ is invertible for $\varepsilon \omega$ small enough \cite{Low freq}. In particular, a $\phi$ solution to
	$$
	\widehat{\mathcal{S}}^{\varepsilon \omega}_D[\phi] = h, \; \text{ for } h \in H^1(\partial B),
	$$
	can be represented as
	$$
	\phi = \phi_0 + \frac{c_1 - \beta_{\varepsilon \omega}}{\beta_{\varepsilon \omega} - c_0} \phi_1,
	$$
	where $(\phi_0,c_0) \in H^1(\partial B) \times \mathbb{R}$ solves 
	$$
	\begin{cases}
	\mathcal{S}_B^0[\phi_0] + c_0 = 0, \\
	\int_{\partial B} \phi_0 = 1,
	\end{cases} 
	$$
	and $(\phi_1,c_1) \in L^2(\partial B) \times \mathbb{C}$ solves
	$$
	\begin{cases}
	\mathcal{S}_B^0[\phi_1] + c_1 = h, \\
	\int_{\partial B} \phi_1 = 1.
	\end{cases} 
	$$
	Moreover, there exists a constant $K$, independent of $\varepsilon$, $\omega$ and $h$, such that
	$$
	\|\phi_0\|_{L^2(\partial B)} + |c_0| \leq K,
	$$
	$$
	\|\phi_1\|_{L^2(\partial B)} + |c_1| \leq K \|h\|_{H^1(\partial B)},
	$$
	that is 
	$$
	\|\phi\|_{L^2(\partial B)} = \|(\widehat{\mathcal{S}}^{\varepsilon \omega}_B)^{-1}[h] \|_{L^2(\partial B)} \leq K \|h\|_{H^1(\partial B)}.
	$$ 
	By solving system \eqref{sys0} for $(\phi, \psi)$, we get
	$$
	\phi= \psi + \frac{1}{2 \pi} \ln(\sqrt{k}) \left(\int_{\partial B} \psi \right) (\widehat{\mathcal{S}}^{\varepsilon \omega}_B)^{-1}[\chi_{\partial B}] + (\widehat{\mathcal{S}}^{\varepsilon \omega}_B)^{-1}[f]
	$$
	and
	$$
	\begin{aligned}
	\psi = \; & \frac{k}{2 \pi (1-k)}\ln(\sqrt{k}) \left(\int_{\partial B} \psi\right) \left(\frac{(k+1)}{2(k-1)}\mathcal{I} - \mathcal{K}^*_B\right)^{-1}\left(-\frac{1}{2}\mathcal{I} +\mathcal{K}^*_B\right) (\widehat{\mathcal{S}}^{\varepsilon \omega}_B)^{-1}[\chi_{\partial B}] \\ & - \frac{k}{1-k}\left(\frac{(k+1)}{2(k-1)}\mathcal{I} - \mathcal{K}^*_B\right)^{-1}\left(-\frac{1}{2}\mathcal{I} +\mathcal{K}^*_B\right)(\widehat{\mathcal{S}}^{\varepsilon \omega}_B)^{-1}[f] + \left(\frac{(k+1)}{2(k-1)}\mathcal{I} - \mathcal{K}^*_B\right)^{-1} [g].
	\end{aligned} 
	$$
	From the above expressions, it is immediate to deduce that
	$$ 
	\|\phi\|_{L^2(\partial B)} + \|\psi\|_{L^2(\partial B)} \leq C(
	\|f\|_{H^1(\partial B)} + \|g\|_{L^2(\partial B)}), 
	$$
	where $C$ does not depend on $\varepsilon$ and $\omega$.
\end{proof}
Now we are in position to prove Proposition \ref{t1}.

\begin{proof}[Proof of Proposition \ref{t1}]
	Let 
	$$\widetilde{\phi}(\tilde{x}) = \phi(\varepsilon \tilde{x} + z), \hspace{3mm} \tilde{x} \in \partial B, $$
	and define $\widetilde{\psi}$, $\widetilde{F}$ and $\widetilde{G}$ likewise. By a change of variables, \eqref{1} reads as follows: 
	\begin{equation*}
	\begin{cases}
	\mathcal{S}_B^{\frac{\varepsilon \omega}{\sqrt{k}}} [\widetilde{\phi}] - \mathcal{S}_B^{\varepsilon \omega}[\widetilde{\psi}] = \varepsilon^{-1}\widetilde{F},\\
	\displaystyle k \frac{\partial \mathcal{S}_B^{\frac{\varepsilon \omega}{\sqrt{k}}} [\widetilde{\phi}]}{\partial \nu} \biggr |_{-} - \frac{\partial \mathcal{S}_B^{\varepsilon \omega}[\widetilde{\psi}]}{\partial \nu} \biggr |_{+} = \widetilde{G},
	\end{cases}
	\text{ on } \partial B.
	\label{3}
	\end{equation*} 
	Consider the operator $T:L^2(\partial B) \times L^2(\partial B) \longrightarrow H^1(\partial B) \times L^2(\partial B)$ defined by
	\begin{equation}
	T(\widetilde{\phi},\widetilde{\psi}) := \left(\mathcal{S}_B^{\frac{\varepsilon \omega}{\sqrt{k}}} [\widetilde{\phi}] - \mathcal{S}_B^{\varepsilon \omega}[\widetilde{\psi}], \left. k \frac{\partial \mathcal{S}_B^{\frac{\varepsilon \omega}{\sqrt{k}}} [\widetilde{\phi}]}{\partial \nu}\right|_{-} - \left. \frac{\partial \mathcal{S}_B^{\varepsilon \omega}[\widetilde{\psi}]}{\partial \nu}\right|_{+}\right) .
	\label{eq:T-operator} 
	\end{equation}
	$T$ can be decomposed as
	\begin{equation*}
	T=T_0+T_{\varepsilon}, 
	\label{5} 
	\end{equation*}
	where
	\begin{equation*}
	T_0(\widetilde{\phi},\widetilde{\psi}) := \left(\mathcal{S}_B^{0} [\widetilde{\phi}] + \beta_{\varepsilon \omega} \int_{\partial B} \widetilde{\phi}- \mathcal{S}_B^{0}[\widetilde{\psi}] - \beta_{\varepsilon \omega} \int_{\partial B} \widetilde{\psi}, \left. k \frac{\partial \mathcal{S}_B^{0} [\widetilde{\phi}]}{\partial \nu}\right|_{-} - \left. \frac{\partial \mathcal{S}_B^{0}[\widetilde{\psi}]}{\partial \nu}\right|_{+}\right),
	\label{6} 
	\end{equation*}
	and 
	\begin{equation*}
	T_{\varepsilon}:=T-T_0.
	\label{7} 
	\end{equation*}
	For $\varepsilon \omega<\varepsilon_0$, and $\varepsilon_0$ small enough, Lemma \ref{t0} implies that 
	\begin{equation*}
	\|T_{\varepsilon} (\widetilde{\phi},\widetilde{\psi})\|_{H^1 \times L^2} \leq C (\varepsilon \omega)^2 \,\ln (\varepsilon \omega)\, ( \| \widetilde{\phi} \|_{L^2} + \| \widetilde{\psi}\|_{L^2}),
	\label{8}
	\end{equation*}
	where $C$ does not depend on $\varepsilon$ nor $\omega$.
	Since $T_0$ is invertible \cite{boundary layer, Low freq, Interface}, $T$ is invertible for $\varepsilon \omega$ small enough and
	\begin{equation*}
	T^{-1} = T^{-1}_0 + E,
	\label{12}
	\end{equation*}
	where the operator $E$ satisfies
	\begin{equation*}
	\|E(\varepsilon^{-1}\widetilde{F},\widetilde{G})\|_{L^2 \times L^2} \leq C (\varepsilon \omega)^2 \,\ln (\varepsilon \omega)\,  \|(\varepsilon^{-1}\widetilde{F},\widetilde{G})\|_{H^1 \times L^2},
	\label{13}
	\end{equation*}
	with $C$ being independent of $\widetilde{F}$, $\widetilde{G}$, $\varepsilon$ and $\omega$. Finally, we have
	\begin{equation*}
	(\widetilde{\phi},\widetilde{\psi}) = T^{-1}(\varepsilon^{-1}\widetilde{F},\widetilde{G}) = T^{-1}_0(\varepsilon^{-1}\widetilde{F},\widetilde{G}) + E(\varepsilon^{-1}\widetilde{F},\widetilde{G}) = (\widetilde{\phi}_0,\widetilde{\psi}_0) + E(\varepsilon^{-1}\widetilde{F},\widetilde{G}) .
	\label{14}
	\end{equation*}
	By applying Lemma \ref{t2}, and assuming $\varepsilon \omega$ small enough, it follows that 
	\begin{equation*}
	\begin{aligned} 
	\|(\widetilde{\phi},\widetilde{\psi})\|_{L^2 \times L^2} & \leq C \|(\varepsilon^{-1}\widetilde{F},\widetilde{G})\|_{H^1 \times L^2}  + C (\varepsilon \omega)^2 \,\ln (\varepsilon \omega)\, \|(\varepsilon^{-1}\widetilde{F},\widetilde{G})\|_{H^1\times L^2} \\ & \leq C \|(\varepsilon^{-1}\widetilde{F},\widetilde{G})\|_{H^1 \times L^2},
	\end{aligned} 
	\label{16}
	\end{equation*}
	where $C$ does not depend on $\varepsilon$ and $\omega$. By scaling back, we get inequality \eqref{2}. 
	
\end{proof}

\vspace{1mm}
	
\section{Frequency-domain asymptotic expansion} 

\label{sec:freq-domain-ae}

Let $D$ be as in the previous section, i.e., $D=\epsilon B +z$. For $x \in \partial D $, $z$ away from the location  $y$ of the source, we consider the truncated Taylor series of the background field 
\begin{equation*}
V_{y, n}(x, \omega) := \sum\limits_{|\alpha|=0}^{n} \frac{\partial^\alpha_z V_y (z, \omega)}{\alpha!} (x-z)^\alpha.
\label{AE:53}
\end{equation*}
Let $(\phi_n, \psi_n) \in L^2(\partial D) \times L^2(\partial D)$ be the unique solution of
\begin{equation}
\begin{cases}
\mathcal{S}_D^{\frac{\omega}{\sqrt{k}}} [\phi_n] - \mathcal{S}_D^{\omega}[\psi_n] = V_{y,n+1},\\
\displaystyle k \frac{\partial \mathcal{S}_D^{\frac{\omega}{\sqrt{k}}} [\phi_n]}{\partial \nu}\biggr |_{-} - \frac{\partial \mathcal{S}_D^{\omega}[\psi_n]}{\partial \nu}\biggr |_{+} = \frac{\partial V_{y,n+1}}{\partial \nu},
\end{cases}
\text{ on } \partial D.
\label{AE:54}
\end{equation}
Then $(\phi - \phi_n, \psi - \psi_n)$ is the unique solution of 
\begin{equation*}
\begin{cases}
\mathcal{S}_D^{\frac{\omega}{\sqrt{k}}} [\phi - \phi_n] - \mathcal{S}_D^{\omega}[\psi - \psi_n] = V_y-V_{y,n+1},\\
\displaystyle k \frac{\partial \mathcal{S}_D^{\frac{\omega}{\sqrt{k}}} [\phi - \phi_n]}{\partial \nu} \biggr |_{-} - \frac{\partial \mathcal{S}_D^{\omega}[\psi - \psi_n]}{\partial \nu} \biggr |_{+} = \frac{\partial (V_y - V_{y,n+1})}{\partial \nu},
\end{cases}
\text{ on } \partial D.
\label{AE:55}
\end{equation*}
By Proposition \ref{t1}, we have
\begin{equation*}
\begin{aligned} 
\|\phi - \phi_n\|_{L^2(\partial D)} + \|\psi - \psi_n\|_{L^2(\partial D)} \leq C (\varepsilon^{-1} \| V_y - V_{y,n+1} \|_{H^1(\partial D)} + \| \nabla (V_y - V_{y,n+1}) \|_{L^2(\partial D)} ),
\end{aligned} 
\label{AE:56}
\end{equation*} 
where $C$ does not depend on $\varepsilon$ and $\omega$. By definition of $V_y - V_{y,n+1}$, we have
\begin{equation*}
\|V_y - V_{y,n+1}\|_{L^2(\partial D)} = \left( \int_{\partial D} |V_y - V_{y,n+1}|^2 \right)^{1/2} \leq |\partial D|^{1/2} \|V_y - V_{y,n+1}\|_{L^{\infty}(\partial D)}.
\label{AE:57}
\end{equation*} 
Hereinafter, we assume that $\omega \in (0, \varepsilon^{-\gamma})$, with $0<\gamma <1$. The following approximation for the Hankel function is needed \cite{Abramowitz, Lebedev, Watson}:
$$
H^{(1)}_n(t) \sim \sqrt{2/(\pi t)} e^{i(t-1/2n\pi -1/4\pi)}, \; \text{ for } t\gg 1.
$$
By expanding $V_y - V_{y,n+1}$ in Taylor series, we obtain
\begin{equation*}
\|\phi - \phi_n\|_{L^2(\partial D)} + \|\psi - \psi_n\|_{L^2(\partial D)} \leq C |\partial D|^{1/2} \varepsilon^{n+1} (1+\omega^{n+3/2}). 
\label{AE:60}
\end{equation*}
For $x \in \mathbb{R}^2 \setminus \overline{D}$, $x \neq y$, $\operatorname{dist}(x,D)\geq c_1>0$, the representation formula \eqref{AE:51} yields 
\begin{equation*}
v(x, \omega) - V_y(x, \omega) = \mathcal{S}^{\omega}_D[\psi_n](x) + \mathcal{S}^{\omega}_D[\psi - \psi_n](x).
\label{AE:61}  
\end{equation*}
By applying the Cauchy-Schwarz inequality, we obtain
\begin{equation*}
\begin{aligned}
|\mathcal{S}^{\omega}_D[\psi - \psi_n](x)| & \leq \left[\int_{\partial D} |\Gamma_{\omega}(x,s)|^2 d \sigma(s)\right]^{1/2} \|\psi-\psi_n\|_{L^2(\partial D)} \\& \leq \|\Gamma_{\omega}(x, \cdot)\|_{L^{\infty}(\partial D)} |\partial D|^{1/2} \|\psi-\psi_n\|_{L^2(\partial D)}. 
\end{aligned} 
\label{AE:62}
\end{equation*}
Then, we have
\begin{equation}
v(x, \omega) - V_y(x, \omega) = \mathcal{S}^{\omega}_D[\psi_n](x) + O(\varepsilon^{n+2}  (|\ln \omega| + \omega^{n+1})).
\label{AE:62bis}
\end{equation}

For each multi-index $\alpha$, define $(\phi_\alpha, \psi_\alpha)$ to be the unique solution to
\begin{equation}
\begin{cases}
\mathcal{S}_B^{\varepsilon \omega }[\phi_\alpha](\tilde{x}) - \mathcal{S}_B^{\frac{\varepsilon \omega }{\sqrt{k}}} [\psi_\alpha](\tilde{x}) = \tilde{x}^\alpha,\\
\displaystyle k \frac{\partial \mathcal{S}_B^{\varepsilon \omega }[\phi_\alpha]}{\partial \nu} \biggr|_{-}(\tilde{x}) - \frac{\partial \mathcal{S}_B^{\frac{\varepsilon \omega }{\sqrt{k}}} [\psi_\alpha]}{\partial \nu} \biggr |_{+}(\tilde{x}) = \frac{\partial \tilde{x}^\alpha}{\partial \nu},
\end{cases}
\tilde{x} \in \partial B,
\label{AE:16bis}
\end{equation}
where $\tilde{x} = \varepsilon^{-1} (x-z)$, $x \in \partial D$. The following proposition has been proved in \cite{reconstruction}.
\begin{prop}
	We claim that
	\begin{equation*}
	\phi_n(x)= \sum\limits_{|\alpha|=0}^{n+1} \varepsilon^{|\alpha|-1} \frac{\partial^\alpha_z V_{y}(z,\omega)}{\alpha!} \phi_\alpha(\varepsilon^{-1} (x-z)),
	\label{AE:17bis}
	\end{equation*}	
	\begin{equation}
	\psi_n(x)= \sum\limits_{|\alpha|=0}^{n+1} \varepsilon^{|\alpha|-1} \frac{\partial^\alpha_z V_{y}(z,\omega)}{\alpha!} \psi_\alpha(\varepsilon^{-1} (x-z)),
	\label{AE:18bis}
	\end{equation}
	for $x \in \partial D$ and $(\phi_n, \psi_n)$ defined as in \eqref{AE:54}.
\end{prop} 
Expansion \eqref{AE:62bis} together with formula \eqref{AE:18bis} yields:
\begin{equation*}
v_y(x, \omega) - V_y(x, \omega) = \sum\limits_{|\alpha|=0}^{n+1} \varepsilon^{|\alpha|-1} \frac{\partial_z^\alpha V_y (z, \omega)}{\alpha!} \mathcal{S}_D^{\omega} [\psi_\alpha(\varepsilon^{-1} (\cdot-z))](x) + O(\varepsilon^{n+2}  (|\ln \omega| + \omega^{n+1})), 
\end{equation*}
for $x \in \mathbb{R}^2 \setminus \overline{D}$ and $x \neq y$. Note that
\begin{equation*}
\mathcal{S}_D^{\omega} [\psi_\alpha(\varepsilon^{-1} (\cdot-z))](x) = \int_{\partial D} \Gamma_{\omega}(x,s) \psi_\alpha(\varepsilon^{-1} (s-z)) \; d\sigma(s) = \varepsilon \int_{\partial B} \Gamma_{\omega}(x, \varepsilon \tilde{s} + z) \psi_\alpha (\tilde{s}) \; d\sigma(\tilde{s}).
\end{equation*}
By a straightforward calculation, we get $\|\Gamma_{\omega}(x,\cdot)\|_{C^{n+2}(\overline{D})} \leq C(1+ \omega^{n+3/2})$, where $C$ does not depend on $\omega$. Therefore, for sufficiently small $\varepsilon$, we have
\begin{equation*}
\Gamma_{\omega}(x,\varepsilon \tilde{s} + z) = \sum\limits_{|\beta|=0}^{n+1} \frac{\varepsilon^{|\beta|}}{\beta!} \partial_z^\beta \Gamma_{\omega} (x,z) \tilde{s}^\beta + O(\varepsilon^{n+1} (1+\omega^{n+3/2})).
\end{equation*}
Finally, we get 
\begin{equation*}
\mathcal{S}_D^{\omega} [\psi_\alpha(\varepsilon^{-1} (\cdot-z))](x) = \sum\limits_{|\beta|=0}^{n+1} \frac{\varepsilon^{|\beta|+1}}{\beta!} \partial_z^\beta \Gamma_{\omega} (x,z) \int_{\partial B} \tilde{s}^\beta \psi_\alpha(\tilde{s}) \; d\sigma(\tilde{s}) + O(\varepsilon^{n+2}(1+\omega^{n+3/2})).
\end{equation*}
For multi-indices $\alpha$ and $\beta$ in $\mathbb{N}^2$, the frequency dependent polarization tensors (FDPTs) $\widehat{W}_{\alpha \beta} := \widehat{W}_{\alpha \beta}(B, \varepsilon \omega, \frac{\varepsilon \omega}{\sqrt{k}})$ are defined as \cite{FDPTs,boundary layer}
\begin{equation} \label{eq:What}
\widehat{W}_{\alpha \beta} := \int_{\partial B} \tilde{s}^\beta \psi_\alpha(\tilde{s}) \; d\sigma(\tilde{s}).
\end{equation}
We obtain the following theorem.

\begin{theo}\label{theo3}
	Suppose that $\omega^2$ is not a Dirichlet eigenvalue for $-\Delta$ on $D$ and $\omega \in (0, \varepsilon^{-\gamma})$, with $0<\gamma <1$. The following asymptotic expansion holds:
	\begin{equation}
	v_y(x, \omega) - V_y(x, \omega) = \sum\limits_{|\beta|=0}^{n+1} \sum\limits_{|\alpha|=0}^{n-|\beta|+1} \frac{\varepsilon^{|\alpha|+|\beta|}}{\alpha! \; \beta!} \partial_z^\alpha V_y (z, \omega) \partial_z^\beta \Gamma_{\omega} (x,z) \widehat{W}_{\alpha \beta} + O(\varepsilon^{n+2}  (|\ln \omega| + \omega^{n+1})), 
	\label{AE:63}
	\end{equation}
	for $x \in \mathbb{R}^2 \setminus \overline{D}$.
	\end{theo}

\vspace{1mm}
	
\section{Time-domain asymptotic expansion} 

\label{sec:tdae}


In this section, we abandon the frequency-domain to inspect our problem in the time-domain.
We define the emitted wave generated at $y \in \mathbb{R}^2\setminus \overline{D}$ as 
	\begin{equation*}
	U_y(x,t) := \frac{H(t-|x-y|)}{2 \pi \sqrt{t^2-|x-y|^2}},
	\label{AE:64}
	\end{equation*} 
	where $H$ is the Heaviside function at $0$ \cite{time-rad}. In particular, $U_y$ satisfies the wave equation
	\begin{equation*}
	\begin{cases}
	(\partial_t^2 - \Delta ) U_y(x,t) = \delta_{x=y} \delta_{t=0}, & (x,t) \in \mathbb{R}^2 \times \mathbb{R},\\
	U_y(x,t)=0 & \text{for } x \in \mathbb{R}^2 \text{ and } t\ll 0.
	\end{cases} 
	\label{AE:65}
	\end{equation*}
	In the presence of a small acoustic inclusion $D$ of contrast $k$ (as described in Section \ref{sec:problem-formulation}), the perturbed wave  $u_y = u_y(x,t)$, is solution to
	\begin{equation*}
	\begin{cases}
	\partial_t^2 u_y - \nabla \cdot (\chi(\mathbb{R}^2 \setminus \overline{D}) + k \chi(D)) \nabla u_y = \delta_{x=y}\delta_{t=0} & \text{ in } \mathbb{R}^2 \times (0, \infty), \\
	u_y(x,t) = 0 & \text{ for } x \in \mathbb{R}^2 \text{ and } t \ll 0.
	\end{cases}
	\label{AE:69}
	\end{equation*}

	For $\rho>0$, we define the operator $P_{\rho}$ acting on tempered distributions by
	\begin{equation}
	P_{\rho}[\psi](t) = \int\limits_{|\omega| \leq \rho} e^{-i \omega t} \widehat{\psi}(\omega) \; d\omega,
	\label{AE:29} 
	\end{equation}
	where $\widehat{\psi}$ is the Fourier transform of $\psi$. The operator $P_{\rho}$ truncates the high-frequency components of $\psi$  \cite{Transient, Numerical Transient}. Note that
	\begin{equation*}
	P_{\rho}[U_y](x,t) = \int\limits_{|\omega| \leq \rho} e^{-i\omega t} \left(\int_{\mathbb{R}} e^{i \omega t} U_y(x,t) \; dt\right) \; d\omega = \int\limits_{|\omega| \leq \rho} e^{-i\omega t} \frac{i}{4} H_0^{(1)}(\omega |x-y|),
	\label{AE:67}
	\end{equation*}
	and satisfies
	\begin{equation*}
	(\partial_t^2 - \Delta ) P_{\rho}[U_y] (x,t) = \delta_{x=y}\psi_{\rho}(t)\quad \mbox{ in } \mathbb{R}^2 \times \mathbb{R}, 
	\label{AE:68}
	\end{equation*}
	where
	$$\psi_{\rho} (t) := \frac{2 \sin \rho t }{t} = \int\limits_{|\omega| \leq \rho}e^{-i \omega t} \; d\omega.$$	

	From Theorem \ref{theo3}, we have
	\begin{equation*}
	\begin{aligned}
	\int\limits_{|\omega| \leq \rho} e^{-i\omega t} (v_y(x,\omega) - V_y(x,\omega)) \; d\omega = & \sum\limits_{|\beta|=0}^{n+1} \sum\limits_{|\alpha|=0}^{n-|\beta|+1} \frac{\varepsilon^{|\alpha|+|\beta|}}{\alpha! \; \beta!} \int\limits_{|\omega| \leq \rho} e^{-i \omega t} \partial_z^\alpha V_y (z, \omega) \partial_z^\beta \Gamma_{\omega} (x,z) \widehat{W}_{\alpha \beta} \; d \omega \\ & + \int\limits_{|\omega| \leq \rho} e^{-i \omega t} R(x,\omega) \; d \omega.
	\end{aligned}
	\label{AE:70}
	\end{equation*}
	Suppose that $\rho = O(\varepsilon^{-\gamma})$ for some $\gamma <1$. Then
	\begin{equation*}
	\begin{aligned}
	 \int\limits_{|\omega| \leq \rho} e^{-i \omega t} R(x,\omega) \; d \omega  = O\left(\varepsilon^{(n+2)(1- \gamma)}\right).
	\end{aligned}
	\end{equation*}
	
	Notice that the following identity holds 
	\begin{equation*}
	\int\limits_{|\omega| \leq \rho} e^{-i \omega t} \partial_z^\alpha V_y (z, \omega) \partial_z^\beta \Gamma_{\omega} (x,z) \widehat{W}_{\alpha \beta} \; d \omega 
	\end{equation*}	
	\begin{equation*}= \int_{\mathbb{R}^2} \partial_z^\alpha P_{\rho}[U_y] (z, t-\tau - \tau') \partial_z^\beta P_{\rho}[U_z] (x,\tau) \left ( \; \int\limits_{|\omega| \leq \rho} e^{-i \omega \tau'} \widehat{W}_{\alpha \beta}(\omega) d \omega  \right ) \; d\tau \; d\tau'.
	\label{AE:72}
	\end{equation*}
	
	This suggests the following definition.
	\begin{definition} \label{def:TDPTs}
		For $\rho < 1/\varepsilon$ and multi-indices $\alpha$ and $\beta$, the two-dimensional truncated time-dependent polarization tensors (hereinafter, TDPTs), $P_{\rho} [W_{\alpha \beta}]$, are defined as:
		\begin{equation}
		P_{\rho} [W_{\alpha \beta}](D,k,t) := \int\limits_{|\omega| \leq \rho} e^{-i\omega t} \widehat{W}_{\alpha \beta}(\omega) \; d\omega,
		\label{2TDPTs}
		\end{equation}
		where $\widehat{W}_{\alpha \beta}$ are the two-dimensional FDPTs.
	\end{definition}	
	\begin{rem} \label{rem:threshold-condition}
		The condition on the truncating threshold $\rho$ in Definition \ref{def:TDPTs}, i.e.,  $\rho < 1/\varepsilon$, boils down to considering only the frequencies $\omega$ for which the FDPTs $\widehat{W}_{\alpha \beta}(\omega)$ are well-defined in the integral transform $P_\rho$.
	\end{rem}	
	\begin{rem} \label{rem:notation}	We warn the reader that the symbol used in Definition \ref{def:TDPTs} for denoting the TDPTs, i.e.,  $P_{\rho} [W_{\alpha \beta}](D,k,t) $, is an abuse of notation, since no $W_{\alpha \beta}$ has been defined, see Remark \ref{rem:threshold-condition}. However, we preferred to keep this notation to remain consistent with the definition of $P_\rho$ given in \cite{Transient, Numerical Transient}.
	\end{rem}
	
	Thus, we have proved the following theorem.
	\begin{theo}\label{theo4}
		For $0<\gamma<1$, the following asymptotic expansion holds:
		\begin{equation}
		\begin{aligned}
		& P_{\rho}[u_y](x,t) = P_{\rho}[U_y](x,t) \\ & + \sum\limits_{|\beta|=0}^{n+1} \sum\limits_{|\alpha|=0}^{n-|\beta|+1} \frac{\varepsilon^{|\alpha|+|\beta|}}{\alpha! \; \beta!} \int_{\mathbb{R}} \partial_z^\beta P_{\rho}[U_z] (x,\tau) \left(\int_{\mathbb{R}} \partial_z^\alpha  P_{\rho}[U_y] (z, t-\tau - \tau') P_{\rho} [W_{\alpha \beta}](\tau') \; d\tau' \right) \; d\tau \\ & + O\left(\varepsilon^{(n+2)(1-\gamma)}\right),
		\end{aligned}
		\label{eq:times-expansion}
		\end{equation}
		where $x \in \mathbb{R}^2 \setminus \overline{D}$, $D=\varepsilon B +z$, $|B|=1$, $P_{\rho} [W_{\alpha \beta}]$ are the TDPTs defined in Definition \ref{def:TDPTs} and $\rho=O(\varepsilon^{-\gamma})$.
\end{theo}

In the end, Theorem \ref{theo4} shows that the scattered wave can be written as a truncated expansion having the TDPTs as building blocks. Since these tensors are the Fourier transformed FDPTs, the transient expansion \eqref{eq:times-expansion} provides a proper interpretation of the multi-frequency problem, which can then be naturally tackled in the temporal domain. 

\vspace{1mm}

\vspace{1mm}

\section{Reconstruction methods}

It is already known that the generalized polarization tensors (GPTs) \cite{GPT, GPT2} and the FDPTs \cite{FDPTs} of an inclusion contain a mixture of geometric information and material parameters. In this section we aim at showing that the same holds for the TDPTs of an acoustic inclusion \eqref{2TDPTs} by extending some of the existing methods that has been established for GPTs and FDPTs.

  Firstly, formulas for determining the size, the contrast and the equivalent ellipse of an inclusion are provided in terms of $P_\rho[W_{\alpha\beta}](t)$. Secondly, the optimal control approach of \cite{GPT, GPT2, FDPTs} for recovering shape details of an inclusion is also adapted in order to perform with the TDPTs. Finally, a procedure for the reconstruction of the TDPTs is presented.
 
 In what follows, without loss of generality, the location of the inclusion is supposed to be known. As a matter of fact, the location can be priorly estimated by using, for instance, a MUSIC-type algorithm, see \cite{MUSIC}.

\subsection{Size, contrast and  equivalent ellipse}  

We now extend the well-known procedure to obtain the equivalent ellipse representing the shape of the inclusion for  $P_\rho[\mathcal{W}_{(1)}](t)$, where 
$$
\widehat{\mathcal{W}}_{(1)} = (\widehat{\mathcal{W}}_{\alpha \beta})_{|\alpha|=|\beta|=1} = \varepsilon^{2} (\widehat{W}_{\alpha \beta})_{|\alpha|=|\beta|=1} .
$$
Since we have 
\begin{equation*} 
P_\rho [\mathcal{W}_{(1)}](D,k,t) \to \frac{2 \sin(\rho t)}{t} M(D,k) \hspace{3mm} \text{as } \varepsilon \to 0,
\label{TDPTs-GPTs}
\end{equation*} 
where $M(D,k)$ is the polarization tensor (PT) (\cite{GPT,math and stat method}), the same procedure of \cite{FDPTs} applies by using $P_\rho[\mathcal{W}_{(1)}](t)$ instead of $\widehat{\mathcal{W}}_{(1)}$.

Let $\mathcal{E}$ be the equivalent ellipse associated to the shape $D$, and let $\theta$ be its rotation angle. Then the size $|D|$ and the contrast $k$ of the inclusion can be estimated as follows. We have
$$
|D|\approx \frac{P_\rho [W_{(0,0),(0,0)}](t)}{\int_{|\omega| \leq \rho} e^{-i\omega t}\omega^2 d\omega},
$$
$$
k \approx \frac{|D|(P_\rho[\mathcal{W}'_{11}] + P_\rho[\mathcal{W}'_{22}]) +(t/(2\sin(\rho t))) P_\rho[\mathcal{W}'_{22}] P_\rho[\mathcal{W}'_{11}]}{(t/(2\sin(\rho t)))P_\rho[\mathcal{W}'_{22}]P_\rho[\mathcal{W}'_{11}] - |D|(P_\rho[\mathcal{W}'_{22}] + P_\rho[\mathcal{W}'_{11}])},
$$
where
$$
\widehat{\mathcal{W}}'_{(1)} = R(-\theta)\widehat{\mathcal{W}}_{(1)}R(-\theta)^T,
$$
with $R(-\theta)$ being the rotation by $-\theta$.  

\medskip

Therefore, from the TDPTs $P_\rho[W_{\alpha\beta}](t)$ we are able to recover an approximation of the volume of the inclusion, separating the information on the material property from the geometric features. 

\subsection{Fine shape details} 

\label{sec:optimization}

So far, we reconstructed the contrast $k$, the size $|D|$ and the equivalent ellipse $\mathcal{E}$. We now reconstruct fine details of the shape of the inclusion using the new concept of high-order TDPTs.

Assuming the inclusion to be a small deformation of the reconstructed equivalent ellipse, we can recover the fine details of its shape by (recursively) minimizing over $D$ the time-dependent discrepancy functional defined by:
\begin{equation}
\begin{aligned}  
J^{(K)}(D)(t)& := \sum_{1\leq|\alpha|+|\beta|\leq K}\left| \sum_{\alpha,\beta}a_\alpha b_\beta P_\rho[W_{\alpha\beta}](D,t) -\sum_{\alpha,\beta}a_\alpha b_\beta P_\rho[W_{\alpha\beta}]^{\text{meas}}\right|^2,
\end{aligned}
\label{discrepancy} 
\end{equation} 
where the coefficients $a_\alpha$, $b_\beta$ are chosen such that $\sum_{\alpha} a_\alpha x^{\alpha}$ and $\sum_{\beta} b_\beta x^{\beta}$ are harmonic polynomials and hence coincide with $cos$ and $sin$ functions on the unit circle.

We introduce the operator $\mathcal{K}_D$ given by
$$
\mathcal{K}_D[\phi](x) = \frac{1}{2 \pi} \int_{\partial D} \frac{\langle y-x, \nu_y \rangle}{|x-y|^2} \phi(y) \ d\sigma(y), \ \ \text{for} \ \phi \in L^2(\partial D).
$$
It is well known that the $L^2$-adjoint of $\mathcal{K}_D$ is 
$$
\mathcal{K}^*_D[\phi](x) = \frac{1}{2 \pi} \int_{\partial D} \frac{\langle x-y, \nu_x \rangle}{|x-y|^2} \phi(y) \ d\sigma(y), \ \ \text{for} \ \phi \in L^2(\partial D).
$$
Recall that, for $\eta$ much smaller than $\varepsilon$, 
\begin{equation*}
\begin{aligned}
& \sum_{\alpha,\beta} a_\alpha b_\beta M_{\alpha \beta}(D_\eta,k) -  \sum_{\alpha,\beta} a_\alpha b_\beta M_{\alpha \beta}(D,k) \approx \eta \int_{\partial D} h(x) \widehat{\phi}_{HF}(x) d\sigma(x),
\end{aligned} 
\end{equation*}
where
$$
\widehat{\phi}_{HF} = (k-1) \left[\left.\frac{\partial v}{\partial \nu}\right|_{-} \left.\frac{\partial u}{\partial \nu}\right|_{-} + \frac{1}{k}\left.\frac{\partial u}{\partial T}\right|_{-}\left.\frac{\partial v}{\partial T}\right|_{-}\right],
$$
and
$$
u= H(x) + \mathcal{S}_D\left[\left(\frac{k+1}{2(k-1)}\mathcal{I} - \mathcal{K}^*_D\right)^{-1} \left[\frac{\partial H}{\partial \nu}\right]\right](x),
$$
$$
v= F(x) +\mathcal{D}_D\left[\left(\frac{k+1}{2(k-1)}\mathcal{I} - \mathcal{K}_D\right)^{-1} \left[F\right]\right] (x),
$$
where $H=\sum_\alpha a_\alpha x^{\alpha}$ and $F=\sum_\beta b_\beta x^\beta$ are defined as above, see \cite{GPT,FDPTs}. Since
$$
\sum_{\alpha,\beta} a_\alpha b_\beta P_\rho [W_{\alpha \beta}](D,k) \to \frac{2 \sin(\rho t)}{t} \sum_{\alpha,\beta} a_\alpha b_\beta M_{\alpha \beta}(D,k) \hspace{3mm} \text{as } \varepsilon \to 0,
$$
where $M_{\alpha \beta}=M_{\alpha \beta}(D,k)$ are the (usual) high-order PTs (\cite{GPT, FDPTs}), we have the following approximation formula: 
\begin{equation}
\begin{aligned}
& \sum_{\alpha,\beta} a_\alpha b_\beta P_\rho [W_{\alpha \beta}](D_\eta,t) -  \sum_{\alpha,\beta} a_\alpha b_\beta P_\rho [W_{\alpha \beta}](D,t) \approx \eta \int_{\partial D_{\text{given}}} h(x) P_\rho[\phi_{HF}](x,t) d\sigma(x).
\end{aligned}
\label{approx-formula} 
\end{equation}
Note that
$$
P_\rho[\phi_{HF}](x,t) =  \frac{2 \sin (\rho t)}{t} \widehat{\phi}_{HF}(x).
$$
Therefore, we modify the initial shape $D^{\text{init}}$ to obtain $D^{\text{mod}}$ by the gradient descent method
$$
\partial D^{\text{mod}} = \partial D^{\text{init}} - \left(\frac{J^{(n)} [D^{\text{init}}]}{\sum_j (\langle d_S J^{(n)}[D^{\text{init}}], \psi_j \rangle)^2} \sum_j \langle d_S J^{(n)}[D^{\text{init}}], \psi_j \rangle \psi_j \right) \nu,
$$
where $\nu$ is the outward unit normal to $D^{\text{init}}$ and $\{\psi_j\}$ is a basis of $L^2(\partial D^{\text{init}})$. The shape derivative of $J^{(n)}[D]$ follows immediately from the approximation formula \eqref{approx-formula} and is given by
\begin{equation*} 
\begin{aligned} 
& \langle d_S J^{(n)}[D], h \rangle_{L^2(\partial D)}
\\
& = \sum_{1\leq |\alpha|+|\beta|\leq K} \left( \sum_{\alpha,\beta} a_\alpha b_\beta P_\rho [W_{\alpha \beta}](D,t) -  \sum_{\alpha,\beta} a_\alpha b_\beta P_\rho [W_{\alpha \beta}](B,t) \right) \left\langle P_\rho[\phi_{HF}],h \right\rangle_{L^2(\partial D)}. 
\end{aligned} 
\end{equation*}

As in \cite{GPT, FDPTs}, we can make the optimization procedure recursively by increasing $K$ to refine the reconstruction of the shape details of the inclusion. At each step, the initial guess for the shape is the result of the previous one. The equivalent ellipse in Section 7.2 provides a good initial guess to begin with. 

\subsection{Reconstruction of the TDPTs from multi-frequency MSR measurements}
We begin by recalling how to reconstruct the FDPTs from multi-static measurements. 

Let $D = \varepsilon B + z$ be a small acoustic two-dimensional inclusion of characteristic size $\varepsilon$, and contrast $k$.
Let us consider two arrays: an array of $M$ transmitters $\{y_1, ... , y_{M}\}$  and another of $N$ receivers $\{x_1, ... , x_{N}\}$, both distributed around the inclusion $D$. For a given frequency $\omega \in  [-\rho, \rho]$, let $\mathcal{A}_{\omega}$ be the corresponding $N \times M$ Multi-Static Response (MSR) matrix. Precisely, the $(i,j)$-th entry of $\mathcal{A}_{\omega}$ is given by
\begin{equation} \label{eq:msr} (\mathcal{A}_\omega)_{i,j} = v_{y_j}(x_i, \omega) - V_{y_j}(x_i, \omega), \quad i \in \{1, ... , N\}, \; j \in \{1, ... , M\}, \end{equation}
that is, the scattered field recorded at the receiver $x_i$, due to the transmitter $y_j$.

As usual, in order to model the error in the measurements, additive gaussian white noise $\mathcal{X}_{\text{noise}}$ is used to contaminate $\mathcal{A}$. We suppose that $\mathcal{X}_{\text{noise}} = \sigma_\text{noise} \mathcal{X}_0$, where $\sigma_\text{noise}$ and $\mathcal{X}_0$ is an $N \times M$ complex random matrix with independent and identically distributed $\mathcal{N}(0,1)$ entries. Hence, the entries of $\mathcal{X}_{\text{noise}}$ are independent complex Gaussian random variables with mean zero and variance $\sigma^2_\text{noise}$.

In view of formula \eqref{AE:63}, each entry of the MSR matrix admits the following expansion
\begin{equation*}
(\mathcal{A}_{\omega})_{i,j} = \mathcal{G}_\omega (x_i,z) \widehat{\mathcal{W}}(\omega) \mathcal{G}_\omega (y_j,z)^T + O(\varepsilon^{n+2}  (|\ln \omega| + \omega^{n+1})),
\end{equation*}
where 
$$
\mathcal{G}_\omega (y,z) = \left ( \frac{1}{\alpha!} \partial^\alpha_z \Gamma_\omega (y,z)\right )_{|\alpha|\le n},
$$
$$
\widehat{\mathcal{W}} = (\widehat{\mathcal{W}}_{\alpha \beta})_{|\alpha|+|\beta|\leq n}= \left(\varepsilon ^{|\alpha|+|\beta|} \widehat{W}_{\alpha \beta}\right)_{|\alpha|+|\beta|\leq n}.
$$ 
Then, the tensor $\widehat{\mathcal{W}}$  can be reconstructed from the measurements $\mathcal{A}_{\omega}$ as the least-squares solution to the following problem
\begin{equation}
\widehat{\mathcal{W}}(\omega)^{\text{meas}} \leftarrow \arg \min_{\mathcal{W}} \|\mathcal{G}_\omega (\,: \,,z) \widehat{\mathcal{W}}(\omega) \mathcal{G}_\omega (\,:\,,z)^T - \mathcal{A}_{\omega} \|_F ,
\label{lsq} 
\end{equation}
where $\|\cdot\|_F$ denotes the Frobenius norm of a matrix, see \cite{FDPTs}. 

Now, by using the reconstructed FDPTs at multiple frequencies in a discrete subset of the interval $[-\rho,\rho]$ we can get an approximation for the TDPTs \eqref{2TDPTs}. More precisely, let the set of sampled frequencies $S_L$ be a uniform discretization of the interval $[-\rho, \rho]$, i.e., 
\[ - \rho = \omega_{-L} < \omega_{-L+1} <  ...  < \omega_{-1} < 0 < \omega_{1} < ... < \omega_{L-1} < \omega_L = \rho , \]
with $\omega_{l+1} - \omega_l = \rho/L$ for every $|l|\le L$.  Then the estimator built on this sampling set of frequencies is obtained by applying the discrete Fourier transform (DFT) 
\begin{equation}
P_{\rho} [\mathcal{W}_{\alpha \beta}](t)^{\text{meas}} := \frac{\rho}{L} \sum\limits_{l = - L}^{L} e^{-i \omega_l t} \widehat{\mathcal{W}}_{\alpha \beta}(\omega_l)^{\text{meas}}\;.
\end{equation}
Such estimator is unbiased, with variance
\begin{equation} \label{eq:variance}
Var(P_{\rho} [\mathcal{W}_{\alpha \beta}]^{\text{meas}}) = \frac{\rho^2}{L^2} \sum\limits_{l = - L}^{L} Var(\widehat{\mathcal{W}}_{\alpha \beta}(\omega_l)^{\text{meas}})\;.
\end{equation}
Since the reminder stated in \eqref{AE:63} is singular at $\omega = 0$, caution is needed when dealing with small frequencies.
In order to get the asymptotic behavior of this dispersion term \eqref{eq:variance} as $L \to +\infty$  we should slightly modify the choice of the range of frequencies by casting a neighborhood of $\omega = 0$ away. In particular for some small $\rho_0 > 0$ we require that $S_L \cap [-\rho_0, \rho_0] = \emptyset$, $S_L$ being uniformly distributed in $[-\rho, -\rho_0]$ and $[\rho_0, \rho]$, separately. Hence $ Var(\widehat{\mathcal{W}}_{\alpha \beta}(\omega_l)^{\text{meas}}) \le C(\rho_0)$ for all $\omega \in S_L$. Then it is readily seen that $Var(P_{\rho} [\mathcal{W}_{\alpha \beta}]^{\text{meas}}) \to 0$ as $L \to +\infty$.

Note that this reconstruction presented here is indirect in the sense that we don't extract the TDPTs directly from the temporal data. Instead, the estimation is done by aggregating the results of multiple reconstructions in the frequency-domain.

\vspace{1mm}

\section{Numerical illustrations} 

In this section, we present some numerical simulations to corroborate the theoretical results of our paper. The simulations aim at showing that the new concept of TDPT can be successfully employed for imaging a small acoustic inclusion. 

In what follows, all the experiments are carried out in the two-dimensional case. First, we perform an analysis of the computational accuracy of $P_\rho [\mathcal{W}]^{\text{meas}}$, which is reconstructed from MSR measurements using the method proposed in Section 7.3. Then we test the optimization procedure in Section 7.2 to restore the fine shape details of the inclusion.

\subsection{Reconstruction of the TDPTs}

We present computational results regarding the reconstruction of TPDTs from MSR measurements by solving \eqref{lsq}. The comparison between $P_\rho[\mathcal{W}]^{\text{meas}}$, obtained from measured $\widehat{\mathcal{W}}$, with $P_\rho[\mathcal{W}]$, numerically computed by solving \eqref{AE:16bis} and computing \eqref{2TDPTs}. For the latter, boundary elements techniques are used in the exact evaluation of $\widehat{\mathcal{W}}$.

Let $D_1$ and $D_2$ be two small acoustic inclusions of same characteristic size $\varepsilon = 0.05$ and contrast $k = 3$, centered at $z_1 = [0.3,-0.1]$ and $z_2 = [0,0.25]$, as shown in Figure \ref{fig:Dsetting}. 

We consider coincident arrays of transmitters and receivers to acquire the multi-static data \eqref{eq:msr}. In particular, circular  and square  configurations are tested in the reconstruction of $\widehat{\mathcal{W}}$ for $D_1$ (Figure \ref{fig:settingcircle}) and $D_2$ (Figure \ref{fig:settingsquare}), respectively.

\begin{figure}[h!]
	\centering
	\hspace*{5mm}
	\begin{subfigure}[b]{0.46\textwidth}
		\centering
		\includegraphics[width=\textwidth]{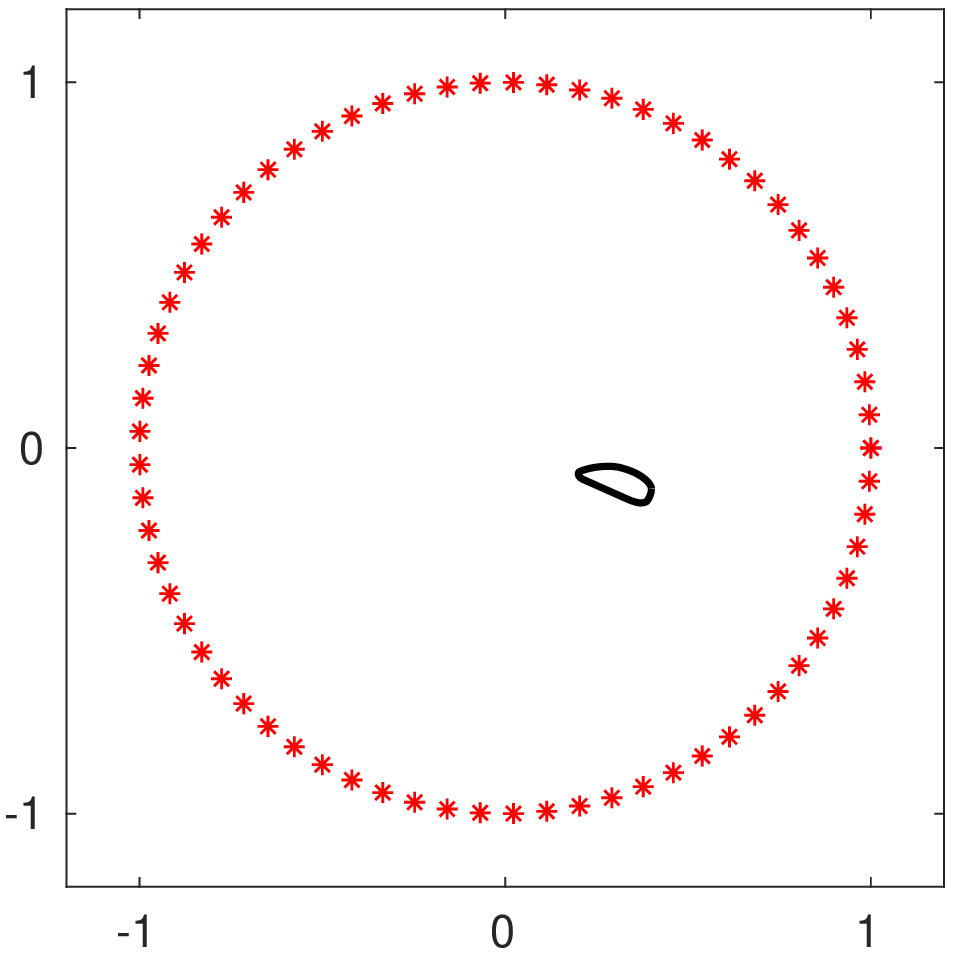}
		\caption{$N_1 = 70$ transmitters/receivers on the unit circle, surrounding the inclusion $D_1 = \varepsilon B_1 + z_1$.}
		\label{fig:settingcircle}
	\end{subfigure}
	\hfill
	\begin{subfigure}[b]{0.47\textwidth}
		\centering
		\includegraphics[width=\textwidth]{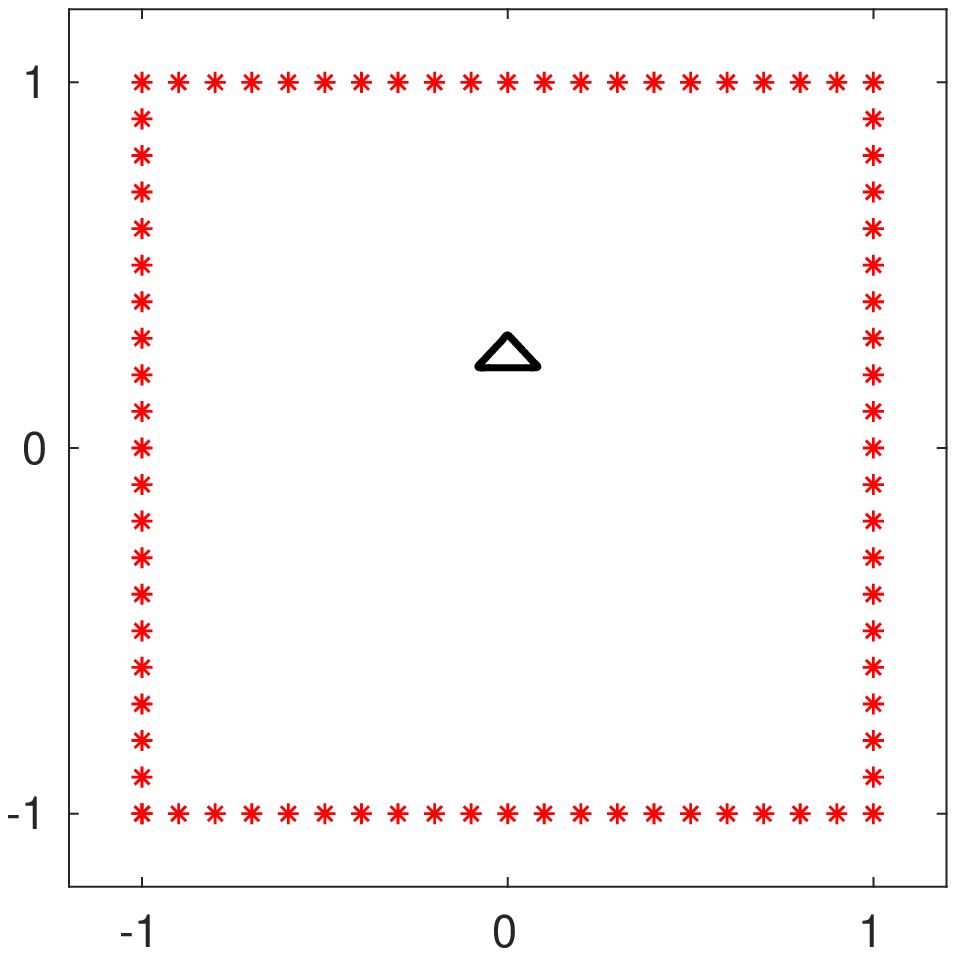}
		\caption{$N_2 = 80$ transmitters/receivers on the unit square, surrounding the inclusion $D_2 = \varepsilon B_2 + z_2$.}
		\label{fig:settingsquare}
	\end{subfigure}
	
	\caption{Geometries of the acquisition setting. The black curves correspond to two small inclusions with common size $\varepsilon = 0.05$ and contrast $k = 3$.}
	\label{fig:Dsetting}
\end{figure}

\begin{figure}[H]
	\centering
	\hspace*{-5mm}
	\begin{tabular}{rl}		
		\includegraphics[scale=0.55]{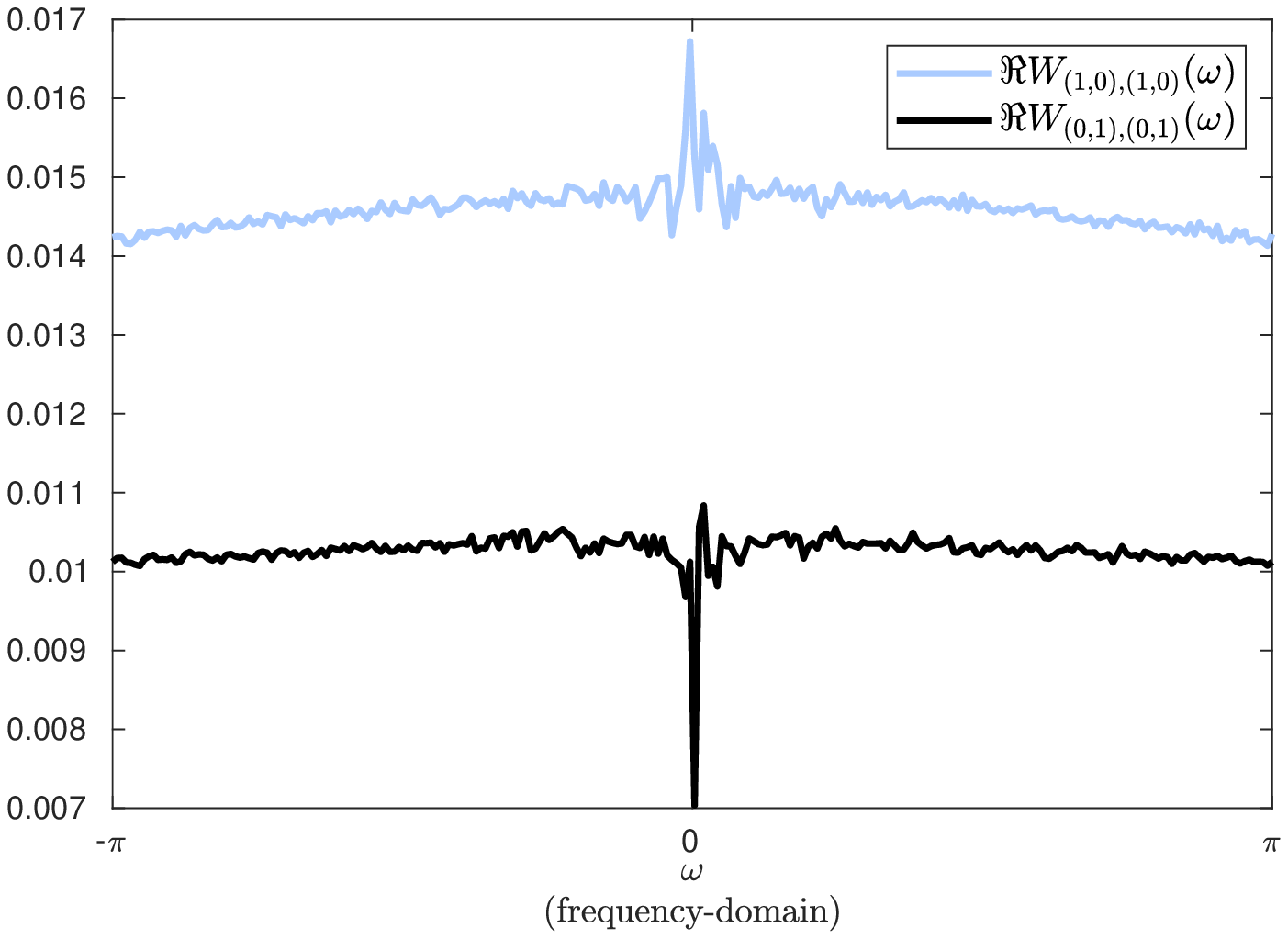}	&	\includegraphics[scale=0.55]{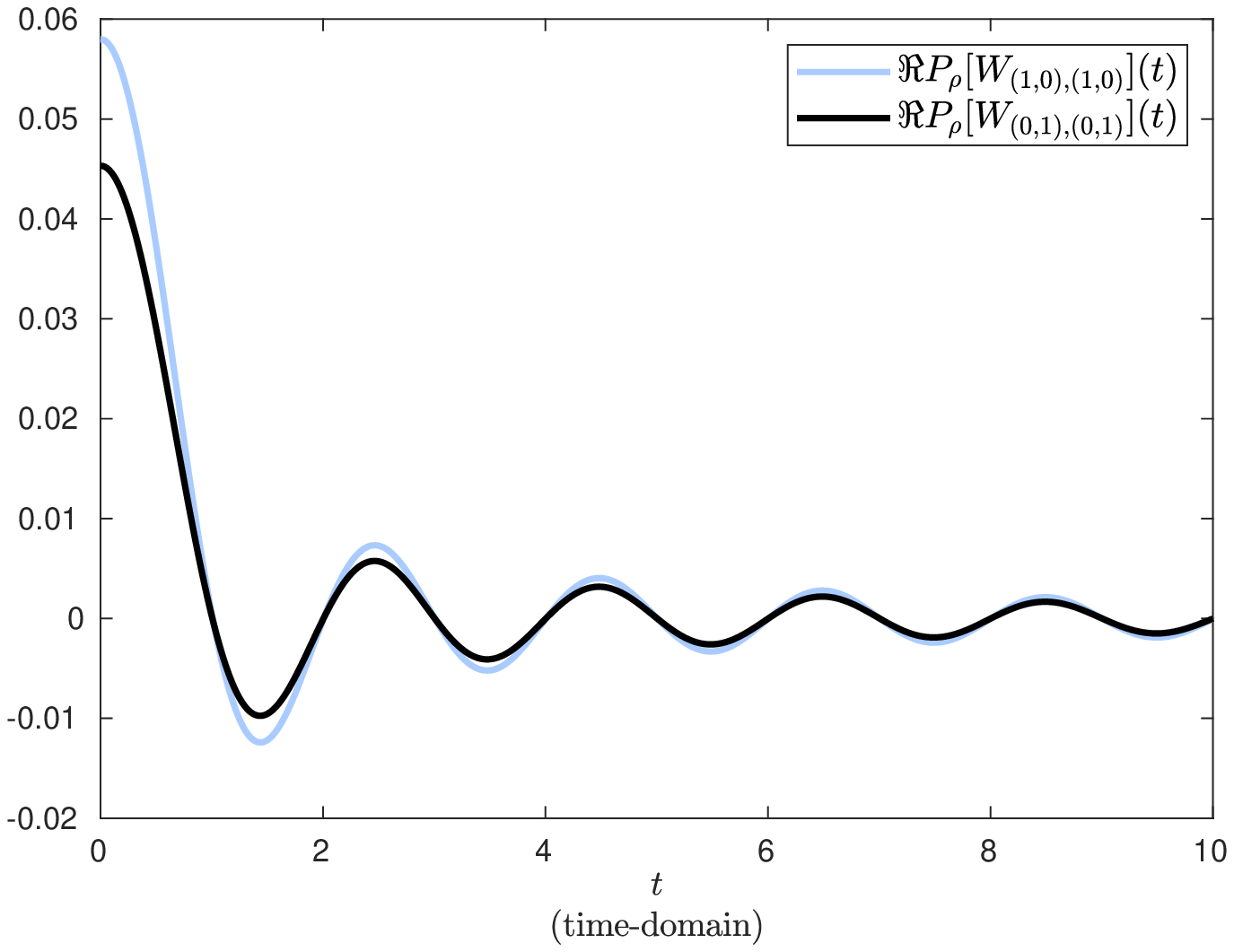} \\
	\end{tabular}
	\caption{The diagonal elements of the reconstructed first order FDPT, namely $\widehat{\mathcal{W}}_{\alpha \beta}$ with $|\alpha| = |\beta| = 1$ (on the left), and the corresponding first order TDPT, namely $P_\rho[\mathcal{W}_{\alpha \beta}]$, over the interval $[0,5]$ (on the right). 20\% of noise is considered in the reconstruction of the FDPT.}
	\label{fig:noisysignals}
\end{figure}

The reconstruction of the first order TDPT of $D_1$ with $20\%$ of noise is reported in Figure \ref{fig:noisysignals}. A uniform sampling of $2^8$ frequencies within $[-\rho, \rho] = [-\pi, \pi]$ is used, and the resulting $P_\rho[\mathcal{W}_{\alpha \beta}]^{\text{meas}}$ is plotted over the interval $[0,5]$. With this choice of $\rho$, the condition on the truncating threshold in Definition \ref{def:TDPTs} is satisfied, see Remark \ref{rem:threshold-condition}.

An analysis of the error in the noiseless reconstruction is performed. Since the TDPTs are functions, the $L^2$-norm is  adequate. We define the absolute and relative-error as follows:
\begin{equation*}
\text{absErr}(T) = \| P_\rho[\mathcal{W}_{\alpha \beta}]^{\text{meas}} - P_\rho[\mathcal{W}_{\alpha \beta}] \|_{L^2(0,T)},
\end{equation*} 
\begin{equation*}
\text{relErr}(T) = \frac{\| P_\rho[\mathcal{W}_{\alpha \beta}]^{\text{meas}} - P_\rho[\mathcal{W}_{\alpha \beta}] \|_{L^2(0,T)}}{\| P_\rho[\mathcal{W}_{\alpha \beta}] \|_{L^2(0,T)}}.
\end{equation*} 
The results are shown in Figure \ref{fig:L2errors}.

\begin{figure}[H]
	\centering
	\hspace*{-5mm}
	\begin{tabular}{cc}
		\includegraphics[scale=0.5]{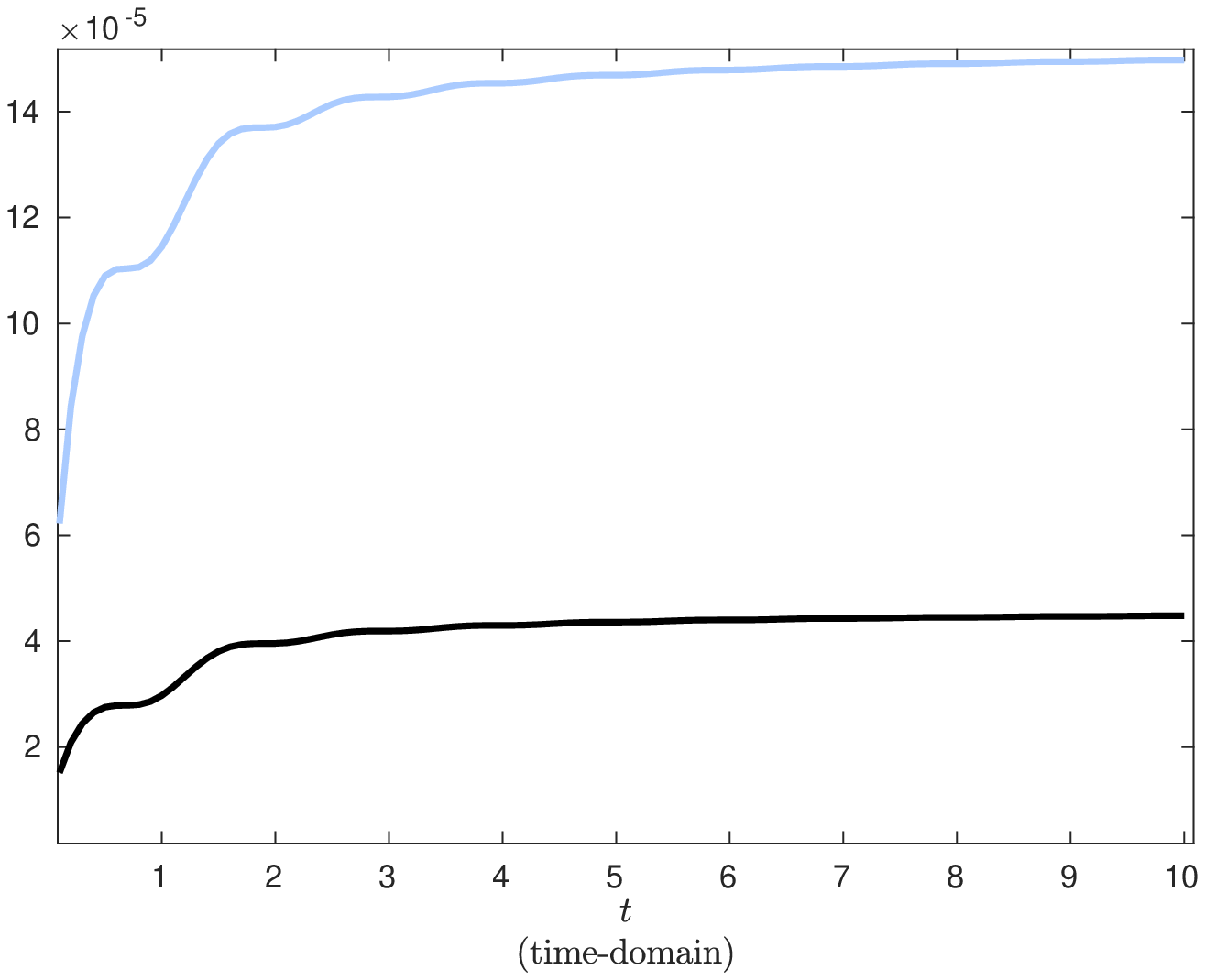} & 
		\includegraphics[scale=0.51]{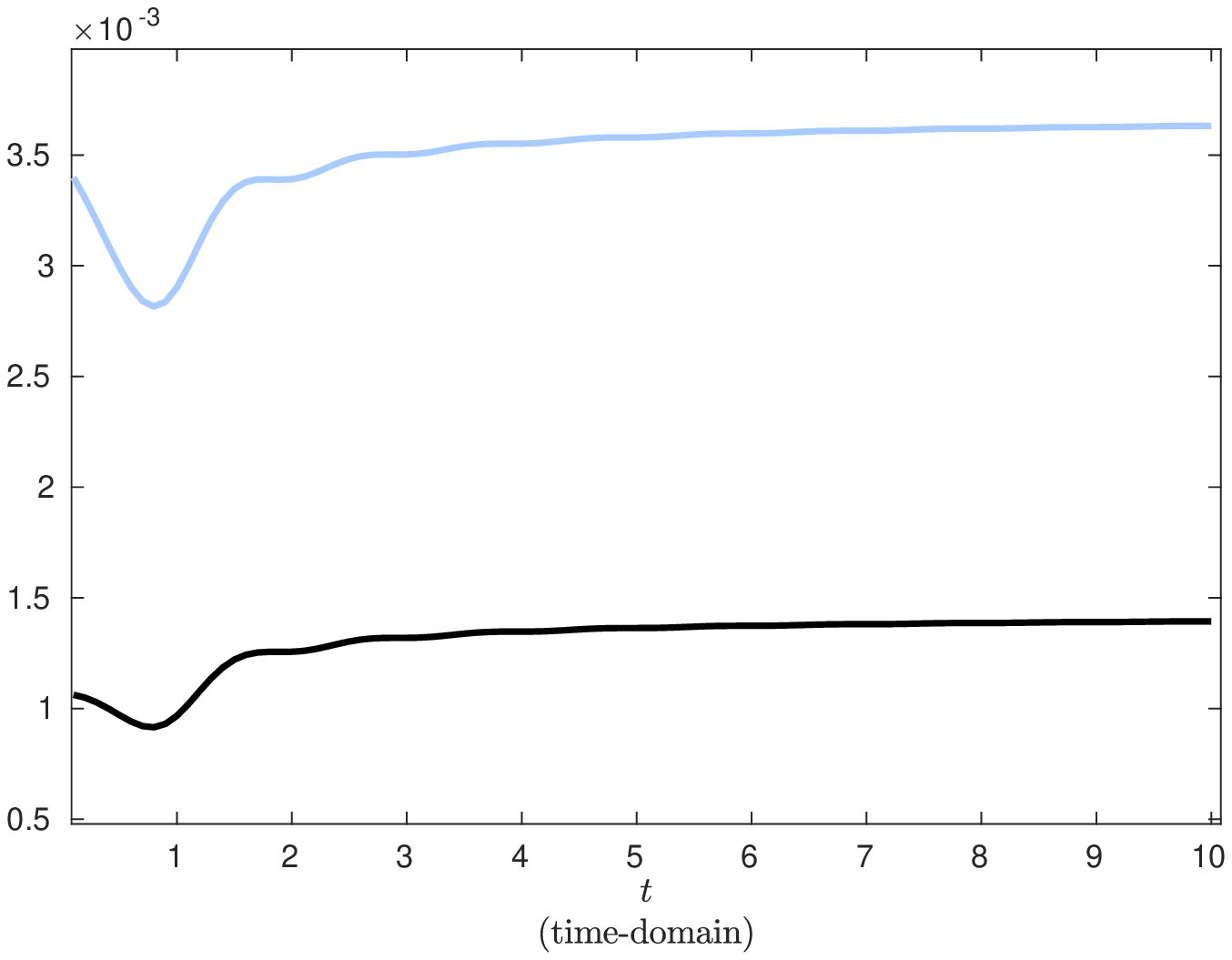}
	\end{tabular}
	\caption{$L^2$ -absolute (left) and -relative (right)  errors in the noiseless reconstruction of $P_\rho [\mathcal{W}_{(1,0),(1,0)}]$ (pale blue) and $P_\rho [\mathcal{W}_{(0,1),(0,1)}]$ (black), assuming the setting shown in Figure \ref{fig:settingsquare}.}
	\label{fig:L2errors}
\end{figure}

\subsection{Reconstruction of the fine shape details}

In this section, we set $k$ to be the value found by the method proposed in Section 7.1, and use the equivalent ellipse as an initial guess for the optimization procedure in Section 7.2.

Firstly, to simulate the reconstruction of the fine details of the inclusions $D_1$ and $D_2$, we reconstruct the TDPTs up to order $n = 4$ by using measurements coming from the two different acquisition settings as in Figure \ref{fig:Dsetting}. A uniform sampling of $2^6$ frequencies within the range $[-\rho, \rho] = [-\pi/8,\pi/8]$ is used to this aim. This is done as described in Section 7.3. After obtaining $P_\rho[\mathcal{W}_{\alpha \beta}]^{\text{meas}}$, we feed them to the optimization algorithm of Section \ref{sec:optimization}. At each step, the algorithm recursively minimizes the discrepancy function \eqref{discrepancy}, yielding a progressive update of the shape. The equivalent ellipse is taken as initial guess. 

The results after few iterations with $20\%$ of noise are shown in Figure \ref{fig:Drec20}. We can observe that details finer than the equivalent ellipse are well recovered for both the inclusions despite the fact that noisy measurements are used to reconstruct the  TDPTs.


\begin{figure}[H]
	\centering
	\hspace*{-5mm}
	\begin{tabular}{ccc}
		\includegraphics[scale=0.41]{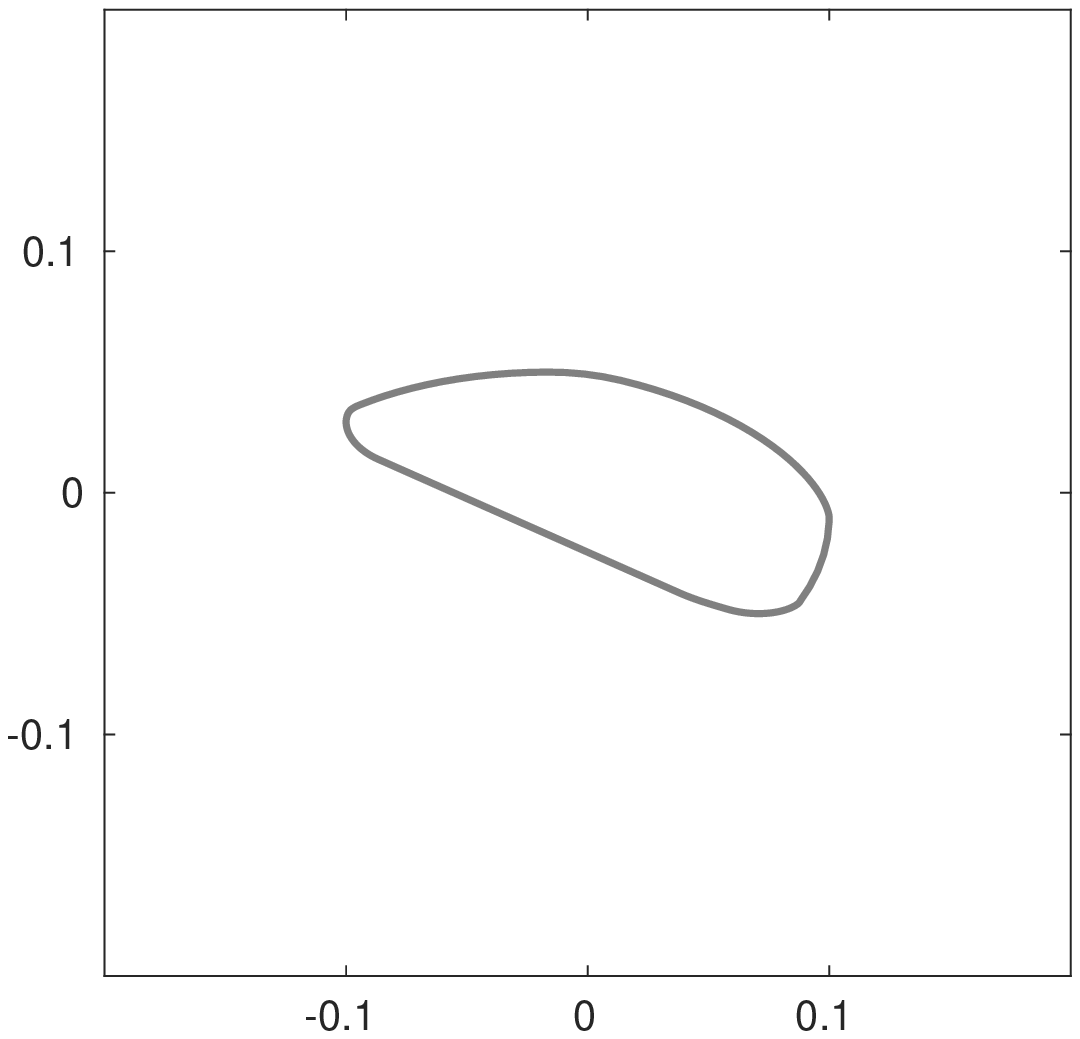} & \includegraphics[scale=0.4]{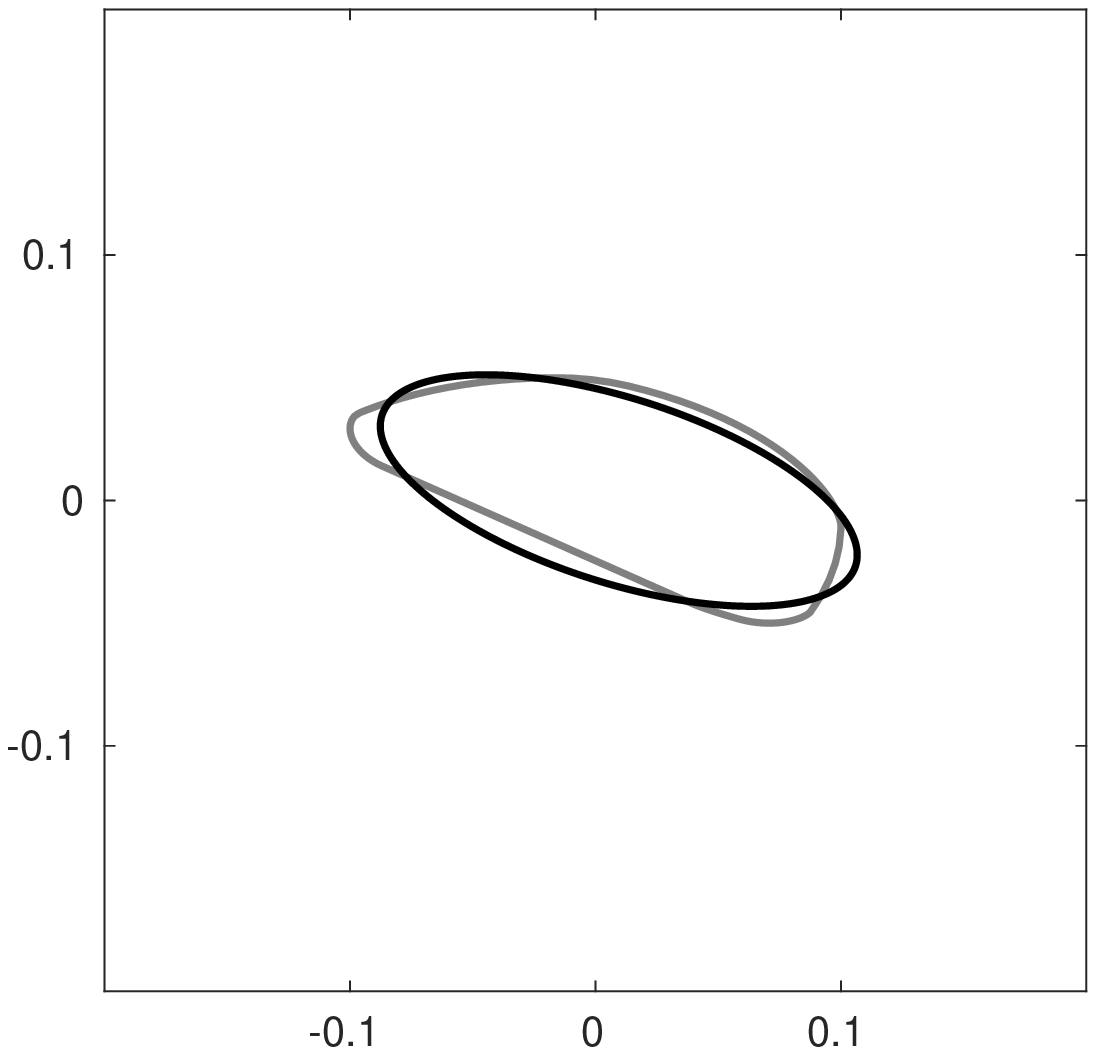} & \includegraphics[scale=0.41]{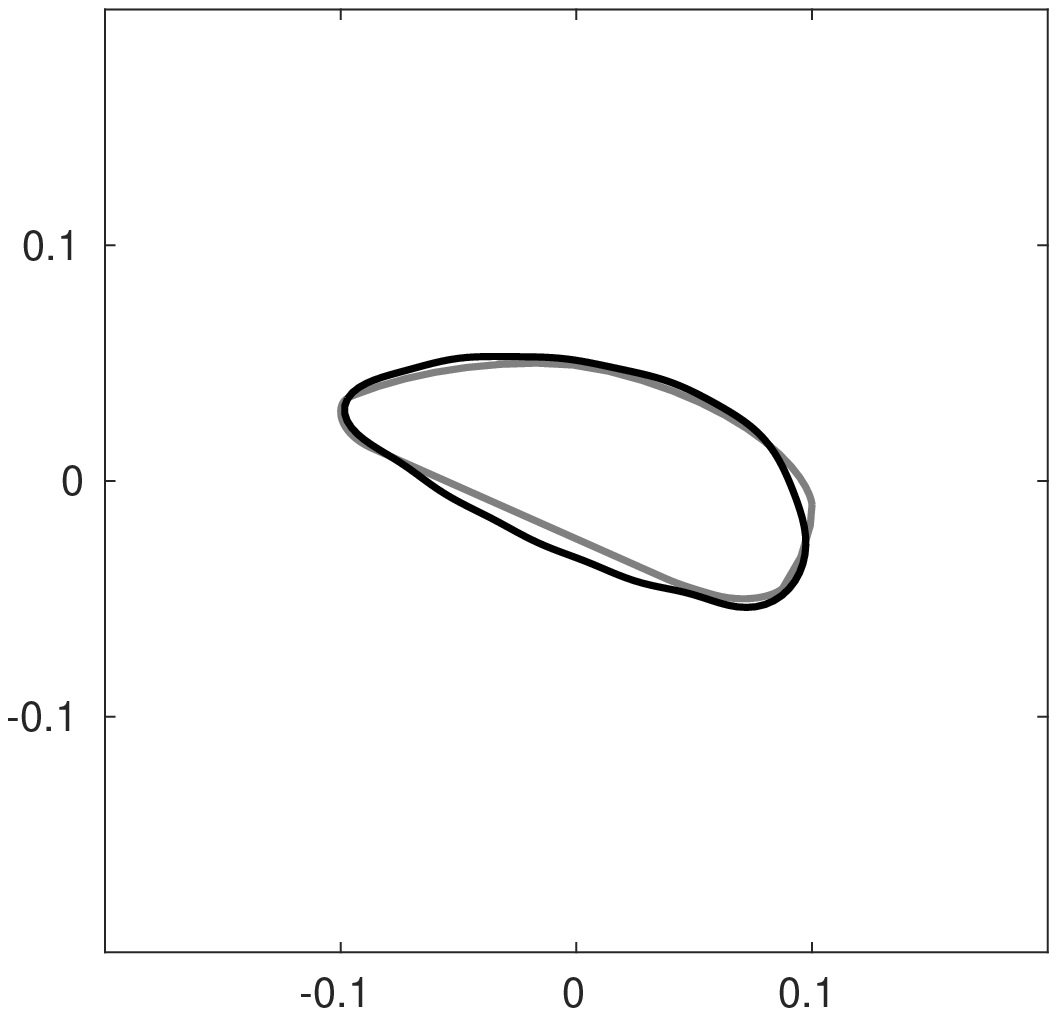} \\
		\includegraphics[scale=0.4]{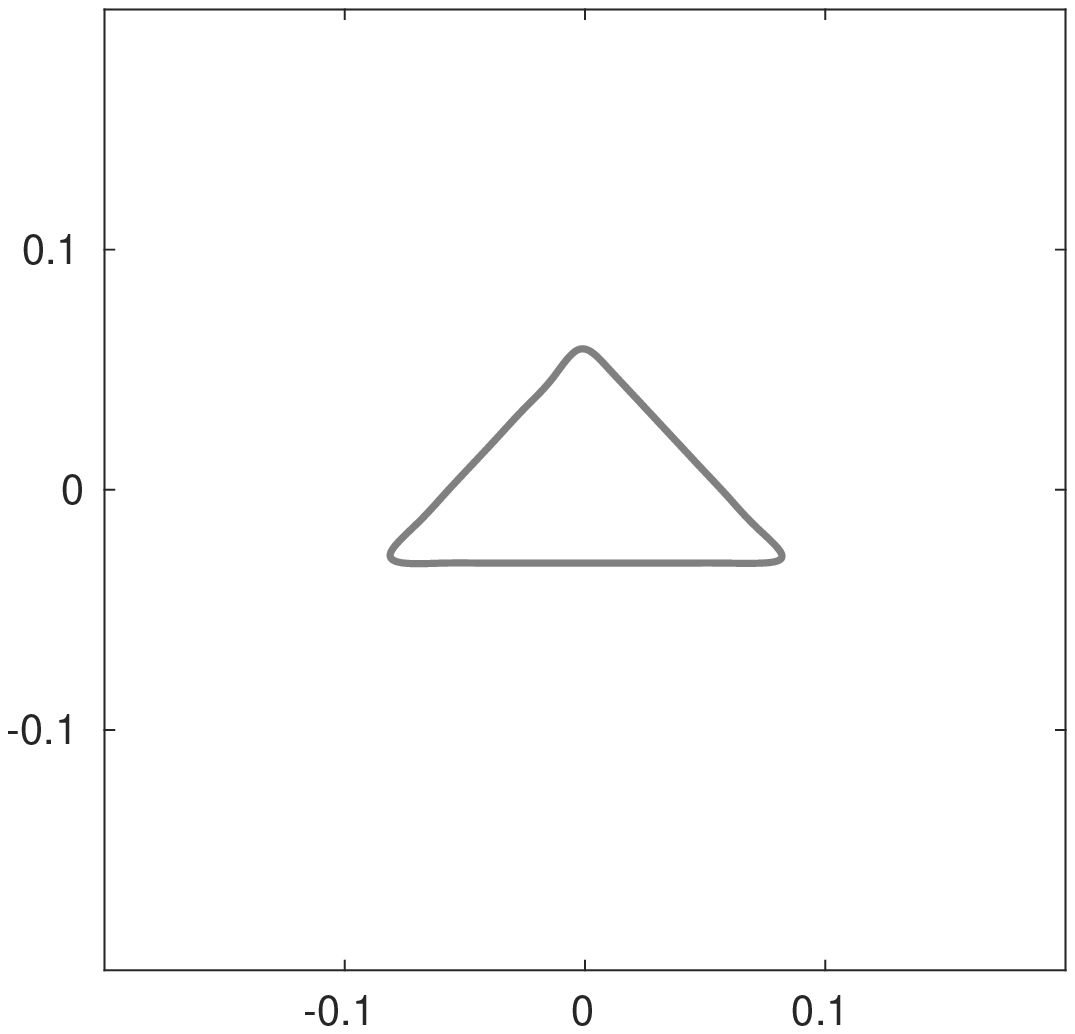} & \includegraphics[scale=0.4]{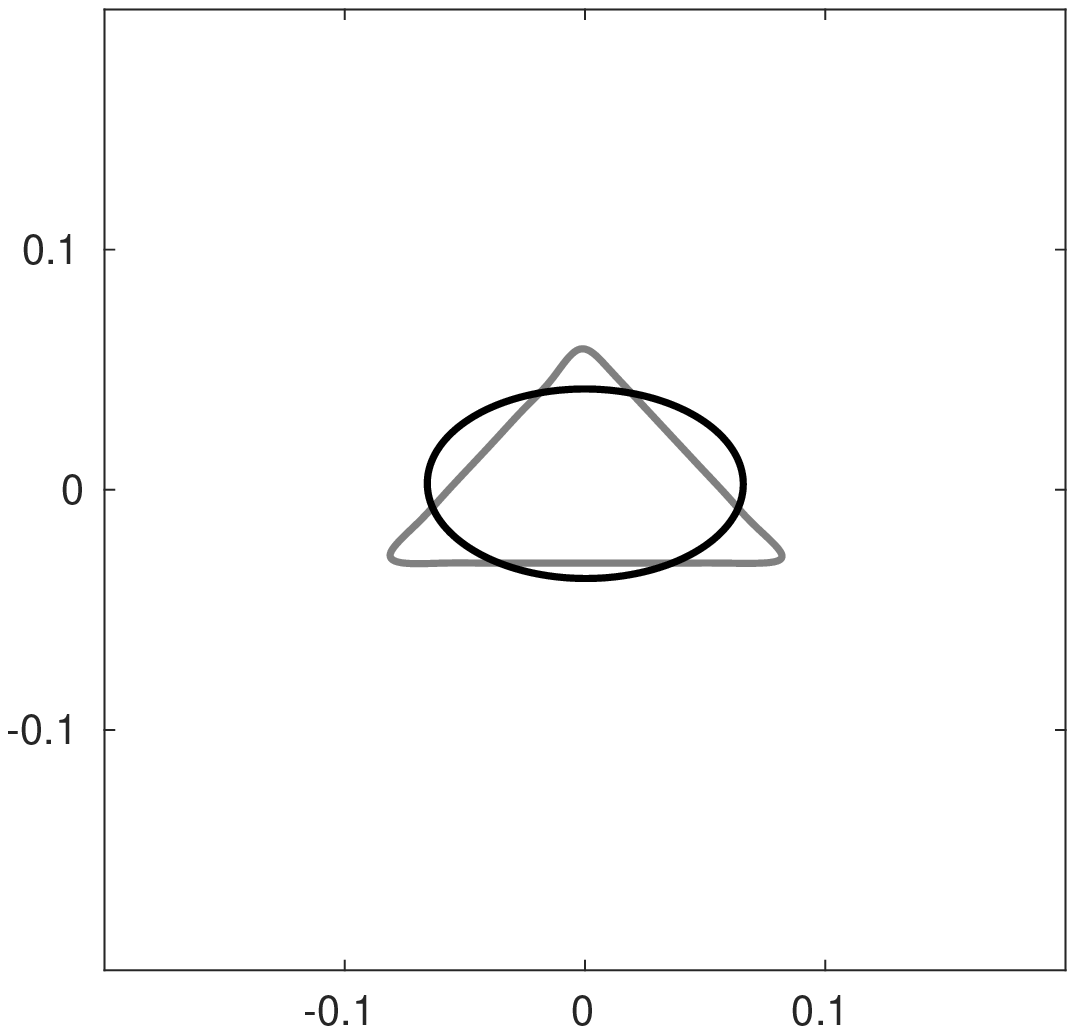} & \includegraphics[scale=0.41]{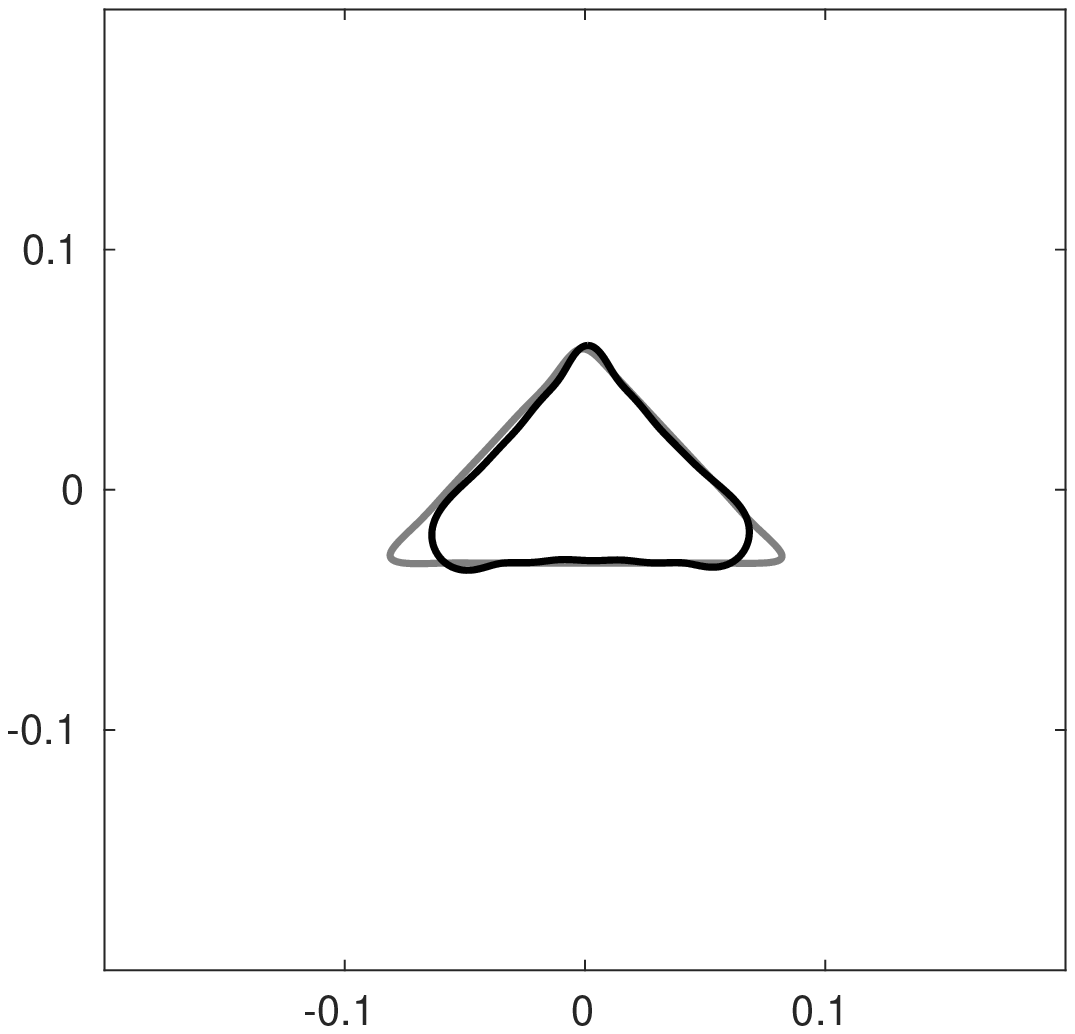} \\
	\end{tabular}
	\caption{Gray curve is the actual inclusion (left column) whereas the black curve represents the equivalent ellipse (mid column) and the reconstructed shape with 20\% of noise after few iterations (right column). }
	\label{fig:Drec20}
\end{figure}

\section{Concluding remarks} 
In this paper, we have presented for the first time the new concept of the TDPTs for the transient problem. These objects are the truncated Fourier transforms of the FDPTs introduced by {Ammari et al.} in \cite{FDPTs}. We have shown that by operating with a range of frequencies, we can recover the high-order TDPTs from the measurements, and this yields a robust reconstruction of the fine shape details of the small acoustic inclusion. 
For future purposes, it is expected that the TDPTs will be relevant to develop promising time-domain techniques for target classification in echolocation by extending the correspondent frequency-domain methods \cite{echolocation, target classification} and the electro-sensing case \cite{fish, electrosensing time, fish-inhom}. A worthwhile extension would be that of considering the case of a bat which uses the movement to better map the surrounding environment.
	
\vspace{1mm}	

\section*{Acknowledgments}
The authors gratefully acknowledge Prof. H. Ammari for his guidance. During the preparation of this work, the authors were financially supported by a Swiss National Science Foundation grant (number 200021-172483).  

\vspace{1mm}

\appendix

\section{The results in the three-dimensional case}\label{appendix1}

For the sake of completeness, in this appendix we focus on the three-dimensional case ($d = 3$). In doing so, we note that the derivation of the asymptotic expansion in time-domain can be carried out similarly to the two-dimensional case (see Sections 4, 5 and 6). Only minor changes in the proofs are required.

In the sections below, we aim at stressing the differences from the proofs which are given for the two-dimensional case.

\subsection{Stability estimates for the Helmholtz equation}

Let $D = \varepsilon B +z$, $|B|=1$ and $D \subset \mathbb{R}^3$. Let $\omega$, $0<\gamma <1$, and $\varepsilon_0 >0$ be as in Section \ref{sec:stability-estimates-2d}. As in the two dimensional case, the main estimate is given by Proposition \ref{t1}. However, while the proof that we presented for the two-dimensional case resort to Lemma \ref{t2}, in the
three-dimensional case it can be proved more directly. We sketch the proof below, highlighting the changes with respect to the previous argument.

\begin{proof} The skeleton of the proof is the same as for the two-dimensional case. The operator $T$ is introduced as in \eqref{eq:T-operator}, and it is decomposed as follows	
	\begin{equation*}
	T=\widetilde{T}_0+\widetilde{T}_{\varepsilon}, 
	\label{5-3} 
	\end{equation*}
	where, this time, $\widetilde{T}_0$ is defined as
	\begin{equation*}
	\widetilde{T}_0(\widetilde{\phi},\widetilde{\psi}) := \left(\mathcal{S}_B^{0} [\widetilde{\phi}] - \mathcal{S}_B^{0}[\widetilde{\psi}], k \frac{\partial \mathcal{S}_B^{0} [\widetilde{\phi}]}{\partial \nu} \biggr |_{-} - 
	\frac{\partial \mathcal{S}_B^{0}[\widetilde{\psi}]}{\partial \nu} \biggr |_{+}\right),
	\label{6-3} 
	\end{equation*}
	and 
	\begin{equation*}
	\widetilde{T}_{\varepsilon}:=T-\widetilde{T}_0.
	\label{7-3}  
	\end{equation*}
	Then, the argument proceeds as in the two-dimensional case.
\end{proof}

\subsection{Frequency-dependent asymptotic expansion}

By using the same techniques of Section \ref{sec:freq-domain-ae}, the frequency-dependent asymptotic expansion in $\mathbb{R}^3$ is readily obtained.

\begin{theo}\label{theo1}
	Suppose that $\omega^2$ is not a Dirichlet eigenvalue for $-\Delta$ on $D$ and $\omega \in (0, \varepsilon^{-\gamma})$, with $0<\gamma <1$. The following asymptotic expansion holds:
	\begin{equation*}
	v_y(x, \omega) - V_y(x, \omega) = \sum\limits_{|\beta|=0}^{n+1} \sum\limits_{|\alpha|=0}^{n-|\beta|+1} \frac{\varepsilon^{|\alpha|+|\beta|+1}}{\alpha! \; \beta!} \partial_z^\alpha V_y (z, \omega) \partial_z^\beta \Gamma_{\omega} (x,z) \widehat{W}_{\alpha \beta} + O(\varepsilon^{n+3} (1+\omega^{n+2})), 
	\label{AE:25}
	\end{equation*}
	for $x \in \mathbb{R}^3 \setminus \overline{D}$, where $\widehat{W}_{\alpha \beta}$ are the FDPTs defined as in \eqref{eq:What}.
\end{theo}

It is worth noticing that the leading-order term of the scattered field derived in \cite{Transient, Numerical Transient} can be recovered from Theorem \ref{theo1}. In particular, we have
\begin{equation*}
v_y(x, \omega) - V_y(x, \omega) = \varepsilon^3 \nabla_z V_y (z, \omega) M(k,B) \nabla_z \Gamma_{\omega} (x,z) + O(\varepsilon^{4}\omega^{3}), 
\label{AE:27}
\end{equation*}
where $M(k,B) = (m_{ij})^3_{i,j}$ is the polarization tensor (PT) given by 
$$
m_{ij} = \int_{\partial B}\xi_j \left(\frac{(k+1)}{2(k-1)}\mathcal{I} - \mathcal{K}^*_B\right)^{-1} [\nu_i](\xi) \; d\sigma(\xi), 
$$  
 $\nu=(\nu_1,\nu_2,\nu_3)$ is the outward unit normal to $\partial B$, $\xi=(\xi_1,\xi_2,\xi_3)$, and $k$ is the contrast. 
\vspace{1mm}

\subsection{Time-domain asymptotic expansion}
In the three-dimensional case, the emitted wave generated at $y \in \mathbb{R}^3\setminus \overline{D}$ is defined as 
$$U_{y}(x,t) := \frac{\delta(t-|x-y|)}{4 \pi |x-y|},$$
where $\delta$ is the Dirac mass at $0$. It is readily seen that $U_y$ satisfies
\begin{equation*}
\begin{cases}
(\partial_t^2 - \Delta ) U_y(x,t) = \delta_{x=y} \delta_{t=0}, & (x,t) \in \mathbb{R}^3 \times \mathbb{R},\\
U_y(x,t)=0 & \text{for } x \in \mathbb{R}^3 \text{ and } t\ll0.
\end{cases} 
\label{AE:28}
\end{equation*} 
For $u_y = u_y(x,t)$, we consider the wave equation
\begin{equation*}
\begin{cases}
\partial_t^2 u_y - \nabla \cdot (\chi(\mathbb{R}^3 \setminus \overline{D}) + k \chi(D)) \nabla u_y = \delta_{x=y}\delta_{t=0} & \text{ in } \mathbb{R}^3 \times (0, \infty), \\
u_y(x,t) = 0 & \text{ for } x \in \mathbb{R}^3 \text{ and } t \ll 0.
\end{cases}
\label{AE:32}
\end{equation*}
For $\rho>0$, let $P_{\rho}$ be the operator defined in \eqref{AE:29}. Since
$$\widehat{U}_y(x,\omega) := \int_{\mathbb{R}} e^{i \omega t} U_y(x,t) \; dt = \frac{e^{i \omega |x-y|}}{4 \pi |x-y|}=V_y (x, \omega) ,$$
it follows that
\begin{equation*}
P_{\rho}[U_y](x,t) = \frac{\psi_{\rho} (t-|x-y|)}{4 \pi |x-y|}, 
\label{AE:30}
\end{equation*}
and $P_{\rho}[U_y]$ satisfies
\begin{equation*}
(\partial_t^2 - \Delta ) P_{\rho}[U_y] (x,t) = \delta_{x=y}\psi_{\rho}(t) \text{ in } \mathbb{R}^3 \times \mathbb{R}. 
\label{AE:31}
\end{equation*}
Moreover, we have that
$$P_{\rho} [u_y] (x,t) = \int\limits_{|\omega| \leq \rho} e^{-i\omega t} v_y(x,\omega) \; d\omega,$$
where $v_y$ is the solution to the three-dimensional problem in the frequency domain.

Similarly to the two-dimensional case, we define the TDPTs as a truncated Fourier transform of FDPTs.
\begin{definition}\label{def:3DTDPTs}
	For $\rho<1/\varepsilon$ and multi-indices $\alpha$ and $\beta$, the three-dimensional TDPTs, $P_{\rho} [W_{\alpha \beta}]$, are defined as follows:
	\begin{equation}
	P_{\rho} [W_{\alpha \beta}](x,t) = \int\limits_{|\omega| \leq \rho} e^{-i\omega t} \widehat{W}_{\alpha \beta} \; d\omega,
	\label{3TDPTs}
	\end{equation}
	where $\widehat{W}_{\alpha \beta}$ are the FDPTs given by \eqref{eq:What}.
\end{definition}
We refer to Remarks \ref{rem:threshold-condition} and \ref{rem:notation}, where the given definition and notation are clarified. 

Proceeding as in Section \ref{sec:tdae}, the following expansion of $P_{\rho}[u_y-U_y](x,t)$ in terms of the TDPTs is readily obtained.
\begin{theo}\label{theo2} 
	For $0<\gamma<1$, the following asymptotic expansion holds:	
	\begin{equation*}
	\label{AE:38}
	\begin{aligned}
	& P_{\rho}[u_y](x,t) = P_{\rho}[U_y](x,t) \\ & + \varepsilon \sum\limits_{|\beta|=0}^{n+1} \sum\limits_{|\alpha|=0}^{n-|\beta|+1} \frac{\varepsilon^{|\alpha|+|\beta|}}{\alpha! \; \beta!} \int_{\mathbb{R}} \partial_z^\beta P_{\rho}[U_z] (x,\tau) \left(\int_{\mathbb{R}} \partial_z^\alpha P_{\rho}[U_y] (z, t-\tau - \tau') P_{\rho} [W_{\alpha \beta}](\tau') \; d\tau' \right) \; d\tau \\ & + O\left(\varepsilon^{(n +3) (1 - \gamma)}\right),
	\end{aligned}
	\end{equation*}
	where $x \in \mathbb{R}^3 \setminus \overline{D}$, $D=\varepsilon B +z$, $|B|=1$, $P_{\rho} [W_{\alpha \beta}]$ are the TDPTs defined in Definition \eqref{def:3DTDPTs} and $\rho=O(\varepsilon^{-\gamma})$.
\end{theo}

\textsc{Department of Mathematics, ETH Z\"urich, R\"amistrasse 101, CH-8092 Z\"urich, Switzerland}

\textit{Email addresses}\textup{\nocorr:\\ \texttt{lorenzo.baldassari@sam.math.ethz.ch}\\
\texttt{andrea.scapin@sam.math.ethz.ch}} 
\end{document}